\documentclass[12pt]{article}
\usepackage[utf8]{inputenc}
\usepackage{geometry}
\usepackage{graphicx}
\usepackage{authblk}

\usepackage{setspace}
\usepackage{amsmath}
\usepackage{amssymb}
\usepackage{amsfonts}
\usepackage{todonotes}
\usepackage{amsthm, mathtools}
\usepackage[authoryear]{natbib}
\usepackage{hyperref}
\usepackage{color}
\usepackage{appendix}
\usepackage{subcaption}

\usepackage{amssymb}
\usepackage{amsmath}

\numberwithin{equation}{section}

\newtheorem{theorem}{Theorem}[section]
\newtheorem{proposition}[theorem]{Proposition}
\newtheorem{corollary}[theorem]{Corollary}
\newtheorem{remark}[theorem]{Remark}
\newtheorem{lemma}[theorem]{Lemma}

\newtheorem{assumption}{Assumption}[section]

\newcommand {\R}{\mathbb{R}}

\newcommand{\diff}{{\rm d}}
\newcommand{\lev}{L\'{e}vy }

\title{Optimal Periodic Double-Barrier Strategies for Spectrally Negative L\'evy Processes}

\author[]{Kazutoshi Yamazaki\thanks{Email: k.yamazaki@uq.edu.au} }
\author[]{Qingyuan Zhang\thanks{(Corresponding author) School of Mathematics and Physics, The University of Queensland, St Lucia, Brisbane, QLD 4072, Australia. Email: qingyuan.zhang@uq.net.au}}
\affil{School of Mathematics and Physics, The University of Queensland}

\date{\today}

\begin{document}

\maketitle
\noindent

\textbf{Abstract.} We study a stochastic control problem where the underlying process follows a spectrally negative L\'evy process. A controller can continuously increase the process but only decrease it at independent Poisson arrival times. We show the optimality of the double-barrier strategy, which increases the process whenever it would fall below some lower barrier and decreases it whenever it is observed above a higher barrier. An optimal strategy and the associated value function are written semi-explicitly using scale functions. Numerical results are also given.



\section{Introduction}
We revisit a stochastic control problem, where the objective is to optimally modify a stochastic process to minimize the expected net present value (NPV) of costs. We consider a two-sided version of this problem, where the process can be modified in both directions---increased or decreased. In this context, the total cost comprises a running cost, represented as a function $f$ of the controlled process accumulated over time, and control costs (or rewards) proportional to the magnitude of the applied control. Such problems have applications in various fields; see \cite{avram_exit_2004, loeffen_optimality_2008} for examples in finance and insurance, and \cite{dai_brownian_2013, dai_brownian_2013-1} for inventory management. 

Specifically, in our problem, the state of a system, for example, the level of inventory in inventory management, is modeled by a spectrally negative L\'evy process (L\'evy process with only downward jumps), where the controller can continuously increase the process but can only decrease it at independent Poisson arrival times. This restriction on downward control to discrete times distinguishes our problem from classical studies of two-sided stochastic control, such as \cite{baurdoux_optimality_2015}, which typically assume that controls in both directions can be applied continuously. Recently, stochastic control problems with random discrete control opportunities have received much attention in the literature; see, for example, \cite{albrecher_optimal_2011, avanzi_periodic_2013, dong_spectrally_2019, noba_optimal_2018, perez_optimal_2020, zhao_optimal_2017}. For recent results on \lev processes observed at Poisson arrival times, see \cite{albrecher_exit_2016, lkabous_poissonian_2021}. 

Restricting downward control opportunities to random discrete times can have useful implications in practice. For instance, in inventory management, selling is often more challenging than replenishing. While replenishment from suppliers may occur continuously, selling surplus stock might not be feasible in continuous time but instead requires some waiting period to find a suitable buyer. In such cases, modeling downward control opportunities with random discrete times is sensible. More generally, the problem we address is applicable in situations where there are tighter constraints on downward control of the stochastic process. 

The motivation for using Poisson arrival times to model control opportunities is two-fold. First, since the inter-arrival times of a Poisson process are exponentially distributed, the solution to the problem can be computed semi-explicitly using scale functions. To the best of our knowledge, explicit analytical solutions fail to exist for other types of discrete random times. In such cases, numerical methods are typically employed, as the state space must be expanded to ensure the problem remains Markovian. Additionally, insights from the mathematical finance literature suggest that the solution to the constant inter-arrival model, where controls can only be applied at deterministic, uniformly spaced times, can be approximated by the solution to the Poissonian inter-arrival model. For a study on this topic, we refer readers to \cite{leung_analytic_2014}.

As is standard in the literature, the running-cost function is assumed to be convex. The aim of this study is two-fold: first, to confirm the optimality of a barrier strategy for the Poisson inter-arrival model, and second, to classify the optimal barrier strategy. We follow the standard guess-and-verify approach to solve this problem:
\begin{enumerate}
    \item The NPV of costs corresponding to the double-barrier strategy is computed semi-explicitly using scale functions. The control cost has been computed in \cite{noba_optimal_2018} and the running cost is computed using a similar technique as in \cite{mata_bailout_2023}. 
    \item The candidate barriers, denoted as $(a^*, b^*)$, are selected using a probabilistic approach mentioned in \cite{noba_stochastic_2022} (see Section \ref{subsection_candidate_barriers}), rather than the standard smooth-fit principle for identifying the candidate barrier(s) in similar problems. Nevertheless, for the current problem, the equivalence of the two approaches is established.
    \item The existence and uniqueness of $(a^*, b^*)$ is established through a probabilistic argument under a most general condition. This finding further intuits previous studies of similar problems.
    \item The optimality of the candidate double-barrier strategy is confirmed via a verification lemma, using a standard argument that leverages the analytical properties of the scale functions and fluctuation identities.
\end{enumerate}

The remainder of this paper proceeds as follows: Section 2 presents the mathematical formulation of the problem considered. Section 3 studies double-barrier strategies and computes the corresponding NPV of costs. Section 4 establishes the optimality of the double-barrier strategy through a verification lemma, followed by a numerical example in Section \ref{Sect: numerical result}. Long, technical proofs and discussions of our running assumptions are deferred to the appendix.

This paper uses $x+$ and $x-$ to denote the right- and left-hand limits at $x \in \mathbb{R}$, respectively. Additionally, terms like strictly increasing and strictly decreasing indicate a function’s strict monotonicity, while non-decreasing and non-increasing indicate weak monotonicity. The convexity of a function is always understood in the weak sense.

\section{Problem formulation}
Defined on a probability space $(\Omega, \mathcal{F}, \mathbb{P})$, let $X = (X(t); t \geq 0)$ be a one-dimensional spectrally negative L\'evy process. For $x \in \mathbb{R}$, we use $\mathbb{P}_x$ to denote the law of $X$ with initial value $x$, and $\mathbb{E}_x$ to denote the corresponding expectation operator. When $x = 0$, we drop the subscript and simply write $\mathbb{P}$ and $\mathbb{E}$. The \emph{Laplace exponent} of $X$ is given by
\begin{align}\label{Eq: laplace exponent}
    \psi(s) \coloneqq \log\mathbb{E} [e^{s X(1)}] = \gamma s + \frac{\sigma^2}{2} s^2 + \int_{(-\infty, 0)} (e^{sz} - 1 - sz 1_{\{z > -1\}}) \,\mu(\diff z), \quad s \geq 0,
\end{align}
for some $\gamma \in \mathbb{R}$, $\sigma \geq 0$ and a \lev measure $\mu$ on $(-\infty, 0)$ satisfying $\int_{(-\infty, 0)} (1 \wedge z^2) \, \mu(\diff z) < \infty$.

The process $X$ has paths of bounded variation if and only if $\sigma = 0$ and $\int_{(-\infty, 0)} (1 \wedge |z|) \,\mu(\diff z) < \infty$. In the case of bounded variation, $X$ admits the form
\begin{equation*}
    X(t) = \delta t - S(t), \quad t\geq 0,
\end{equation*}
where $(S(t); t\geq0)$ is a subordinator and
\begin{align}
    \delta \coloneqq \gamma - \int_{(-1, 0)} z\, \mu(\diff z). \label{Eq: drift finite variation}
\end{align}
We assume that $X$ is not the negative of a subordinator, and therefore $\delta > 0$.

In this paper, we consider the following stochastic control problem. Let $\mathcal{T}_r \coloneqq (T(i); i \geq 0)$ be the arrival times of a Poisson process $N^r = (N^r(t); t \geq 0)$ with intensity $r > 0$. The Poisson process $N^r$ and its corresponding arrival times $\mathcal{T}_r$ are independent of $X$. Let $\mathbb{F} \coloneqq (\mathcal{F}(t); t \geq 0)$ denote the (completed) filtration generated by $(X, N^r)$. A strategy $\pi = \{(R^\pi(t), L^\pi(t)); t \geq 0\}$ is a pair of non-decreasing, c\`adl\`ag, and $\mathbb{F}$-adapted processes, with $R^\pi(0-) = L^\pi(0-) = 0$ and an extra restriction on $L^\pi$ as described below. Under $\pi$, the controlled process $Y^\pi = (Y^\pi(t); t \geq 0)$ is given by
\[Y^\pi(t) \coloneqq X(t) + R^\pi(t) - L^\pi(t), \quad t \geq 0.\]

Specifically, the process $R^\pi$ controls the state process in the upward direction and can be activated continuously in time. In contrast, $L^\pi$ controls the state process in the downward direction, and can only be activated at the arrival times $\mathcal{T}_r$ of the Poisson process $N^r$. Mathematically, $L^\pi$ admits the form
\[L^\pi(t) = \int_{[0, t]}\nu^\pi(s) \, \diff N^r(s) = \sum^\infty_{i = 1} \nu^\pi(T(i)) 1_{\{T(i) \leq t\}},\]
for an $\mathbb{F}$-adapted c\`agl\`ad process $(\nu^\pi(t); t \geq 0)$.
    
Fix a discount factor $q > 0$ and an initial value $x \in \mathbb{R}$ for the state process. Associated with each admissible strategy $\pi$, the NPV of costs is given by
\begin{equation}\label{Eq: NPV of costs}
    v^\pi(x) = \mathbb{E}_x \left[\int^\infty_0 e^{-qt} f(Y^\pi(t))\, \diff t + \int_{[0, \infty)} e^{-qt}(C_U \, \diff R^\pi(t) + C_D \, \diff L^\pi(t))\right].
\end{equation}
Here, $f: \mathbb{R} \to \mathbb{R}$ is a function modeling the running cost, and $C_U$ and $C_D$ are real numbers representing the unit costs (if positive) or rewards (if negative) of controlling. For our problem to be well-defined, we assume
\begin{align}
    C_U + C_D > 0. \label{Eq: sum C}
\end{align}
Assumption \eqref{Eq: sum C} is standard in the literature, such as in \cite{baurdoux_optimality_2015}.

We impose the following assumptions.
\begin{assumption}\label{asm: on f}
    We assume that $f: \mathbb{R} \to \mathbb{R}$ satisfies the following.
    \begin{itemize}
        \item[(1)] $f$ is a convex, piecewise continuously differentiable function. Additionally, it is assumed to be slowly or regularly varying as $x \to \infty$ (resp., $x \to -\infty$) if $\lim_{x \to \infty} |f(x)| = \infty$ (resp., $\lim_{x \to -\infty} |f(x)| = \infty$).
        \item[(2)] There exists a number $\bar{a} \coloneqq \inf\{a \in \mathbb{R}: \tilde{f}'(a) \coloneqq f'(a) + qC_U \geq 0\} \in \mathbb{R}$ such that the function
        \begin{equation} \label{f_tilde_def}
            \tilde{f}(x) \coloneqq f(x) + qC_Ux, \quad x\in \mathbb{R},
        \end{equation}
        is strictly decreasing on $(-\infty, \bar{a})$ and non-decreasing on $(\bar{a}, \infty)$.
        \item[(3)] There exists a number $\bar{\bar{a}} \coloneqq \inf\{a \in \mathbb{R}: f'(a) - qC_D > 0\} \in \mathbb{R}$. 
    \end{itemize}
\end{assumption}

In Assumptions \ref{asm: on f}, the most crucial assumption is convexity, which is commonly made in similar stochastic control problems, such as in \cite{baurdoux_optimality_2015} and \cite{perez_optimal_2020}. Regarding the other technical conditions, the growth condition is imposed for integrability, while (2) and (3) are imposed to avoid the cases where it is optimal not to control the process from below or above. In \cite{baurdoux_optimality_2015}, condition (2) is imposed but  (3) is not explicitly specified. We discuss the case where (3) fails to hold in Appendix \ref{Sect: assumption on f (3)}, showing that in such cases it is optimal not to activate the downward control.

Every polynomial with degree $n \geq 1$ is regularly varying at $\pm \infty$. Hence, any convex, piecewise continuously differentiable function $f$, where $f$ is a polynomial of degree $n \geq 1$ outside a compact interval, satisfies Assumption \ref{asm: on f}(1).

\begin{remark}\label{Remark: polynomial growth}
    By Assumption \ref{asm: on f}, $f$ grows at most polynomially, and the derivatives satisfy $f'(-\infty) \coloneqq \lim_{x \to -\infty} f'(x) \in [-\infty, -qC_U)$ and $f'(\infty) \coloneqq \lim_{x \to \infty} f'(x) \in (qC_D, \infty]$.
\end{remark}

Moreover, by \eqref{Eq: sum C} and the convexity of $f$, we have $\overline{a} < \bar{\bar{a}}$. We assume the following for integrability for the \lev process $X$. 
\begin{assumption} \label{asm: on X}
    There exists $\theta > 0$ such that $\int_{(-\infty, -1]} \exp(\theta |y|)\, \mu(\diff y) < \infty$, which guarantees $\mathbb{E}(X(1)) = \psi'(0+) \in (-\infty, \infty)$ and $\mathbb{E}_x[\int^\infty_{0}e^{-qs} |f(X(s))|\, \diff s] < \infty$, for all $x \in \mathbb{R}$, by Assumption \ref{asm: on f}(1) and Remark \ref{Remark: polynomial growth}. 
\end{assumption}

We call a strategy $\pi$ \textit{admissible} if it further satisfies the following integrability conditions: (1) $\mathbb{E}_x[\int^\infty_{0}e^{-qs} |f(Y^\pi(s))|\, \diff s] < \infty$, and (2) $\mathbb{E}_x[\int_{[0, \infty)}e^{-qs} (\diff R^\pi(s) + \diff L^\pi(s))] < \infty$. Let $\Pi$ be the set of all admissible strategies.
 The objective of this stochastic control problem is to derive an optimal value function
\[
v(x) \coloneqq\inf_{\pi \in \Pi} v^\pi(x), \quad x \in \mathbb{R},
\]
and an associated optimal strategy $\pi^*$ such that $v(x) = v^{\pi^*}(x)$, if such a strategy exists.

\section{Periodic-classical barrier strategies}\label{Sect: contruction of process}
The objective of this paper is to show the optimality of a \textit{periodic-classical barrier strategy} 
\[\pi_{a,b} \coloneqq \{(L^{a,b}(t), R^{a,b}(t)); t \geq 0 \},\]
which is parameterized by $a,b \in \mathbb{R}$ with $a < b$. The controlled process 
\[ Y^{a,b}(t) \coloneqq X(t) + R^{a,b}(t) - L^{a,b}(t), \quad t \geq 0, \]
is continuously reflected from below at the lower barrier $a$ and periodically reflected from above at the upper barrier $b$. In other words, the controlled process is never allowed to fall below $a$ and is thus reflected upwards whenever it would fall below this barrier. In contrast, at any decision time in $\mathcal{T}_r$, if the controlled process is observed above $b$, it is immediately pushed down to $b$. In particular, for the case $b = \infty$, $Y^{a,\infty}$ reduces to the L\'evy process reflected from below in the classical sense. The process $Y^{a, b}$ has been studied in works such as \cite{noba_optimal_2018}, among others. As described in \cite{noba_optimal_2018}, it can be defined as a concatenation of classical reflected processes with additional downward jumps. An example is provided in Figure \ref{Fig: double-barrier}. 
\begin{figure}[htbp]
\centering
\begin{subfigure}[b]{0.49\textwidth}
    \centering
    \includegraphics[width=\textwidth]{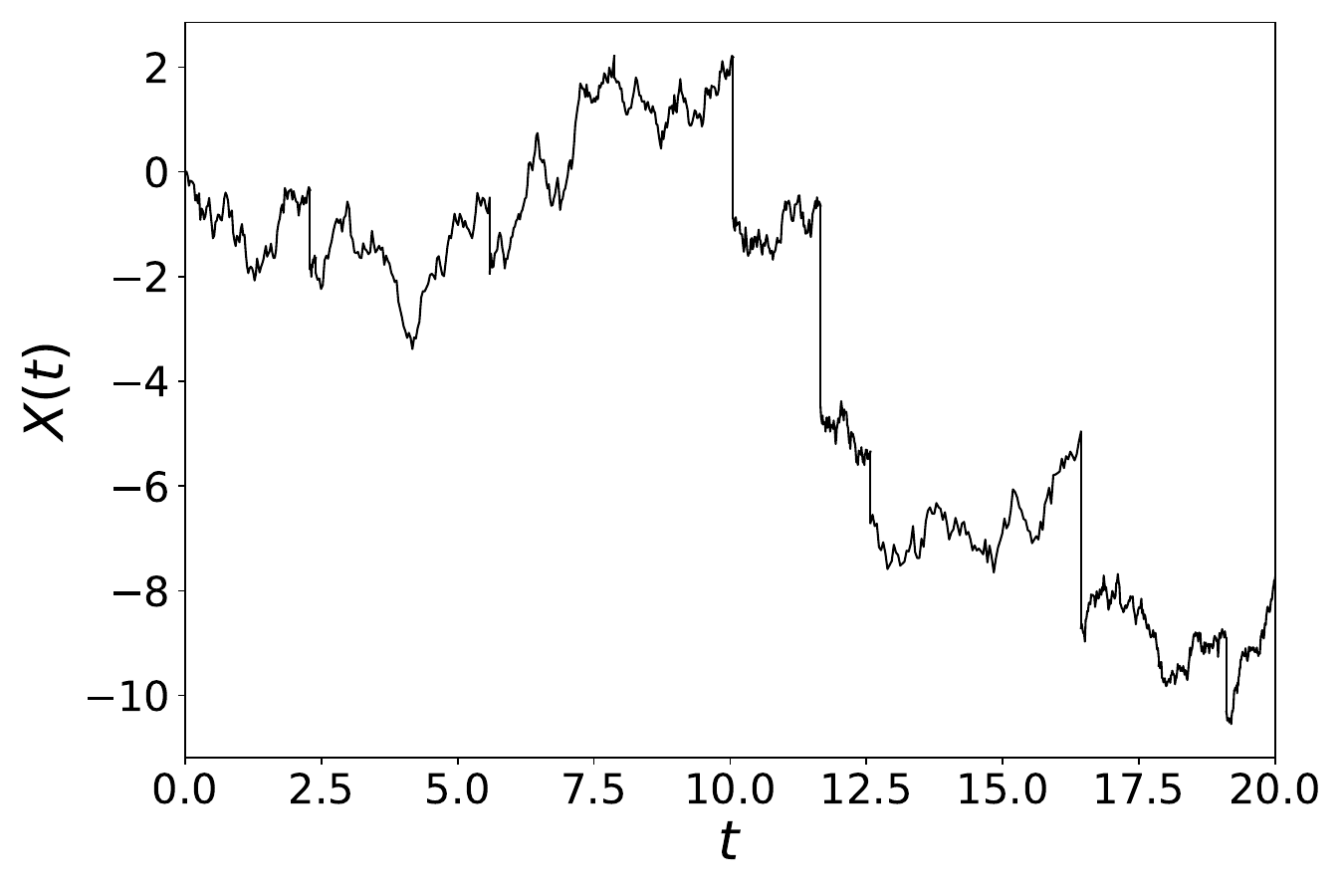}
\end{subfigure}
\begin{subfigure}[b]{0.49\textwidth}
    \centering
    \includegraphics[width=\textwidth]{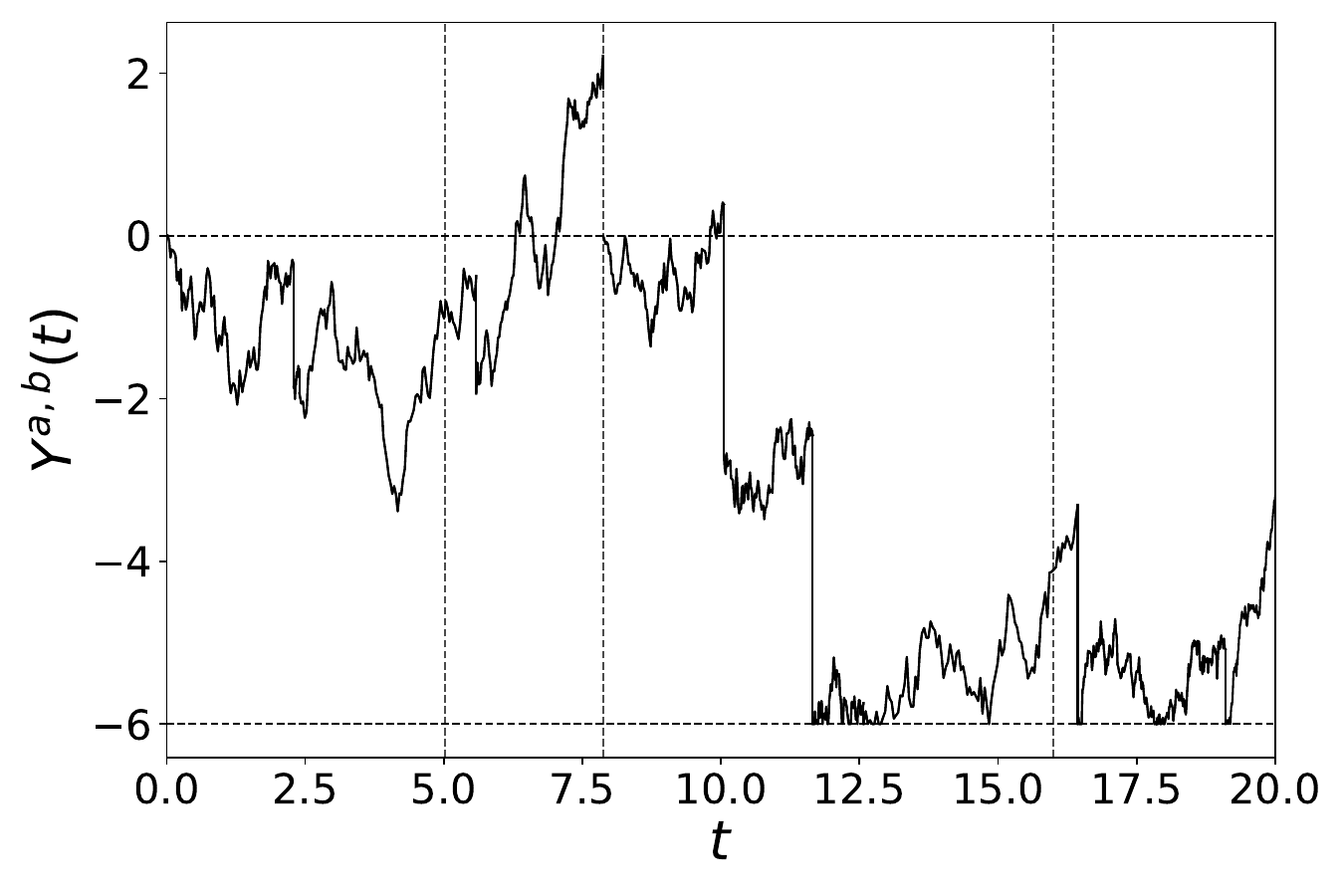}
\end{subfigure}
\vspace{-10pt}
\caption{Left: Sample path of a spectrally negative L\'evy process $X$. Right: Corresponding controlled process $Y^{a, b}$ with $a = -6$ and $b = 0$. Vertical dashed lines indicate the Poisson arrival times $\mathcal{T}_r$. Among the three arrival times displayed, $Y^{a, b}$ exceeds $b = 0$ only at the second one, causing downward control to activate there but not at the first or third arrival times.}
\label{Fig: double-barrier}
\end{figure}

\subsection{Scale functions and fluctuation identities}
The NPV of costs under $\pi_{a,b}$ can be written in terms of scale functions, as obtained in \cite{mata_bailout_2023} and \cite{perez_optimal_2020}. For $q \geq 0$, the scale function $W^{(q)}: \mathbb{R} \to [0 , \infty)$ of the spectrally negative L\'evy process $X$ is defined as follows. On the positive half-line, it is a continuous and strictly increasing function that satisfies
\begin{align}\label{Eq: scale function laplace}
    \int_0^\infty e^{-\theta x} W^{(q)}(x)\, \diff x = \frac{1}{\psi(\theta) - q}, \quad \theta > \Phi_q. 
\end{align}
Here, $\psi$ is the Laplace exponent defined in \eqref{Eq: laplace exponent} and
\begin{align*}
\begin{split}
    \Phi_q \coloneqq \sup \{s \geq 0: \psi(s) = q\}, \quad q \geq 0,
\end{split}
\end{align*}
is its right inverse. On the negative half-line, $W^{(q)}$ is set to be zero. We refer the reader to \cite{kuznetsov_theory_2013} and \cite{kyprianou_fluctuations_2014} for a comprehensive review of scale functions. In this paper, the following smoothness properties of scale functions are of particular importance.
\begin{remark} \label{remark_smoothness_scale_function} By Lemmas 2.1, 2.4, and 3.1 of \cite{kuznetsov_theory_2013}, 
\begin{enumerate}
    \item[(1)] $W^{(q)}$ is differentiable almost everywhere (a.e.) on $\mathbb{R}$.  
    \item[(2)] If $X$ is of unbounded variation, or if the L\'evy measure does not have an atom, then $W^{(q)}$ is $C^1(\mathbb{R} \backslash {\{0 \}})$.
    \item[(3)] The value of $W^{(q)}$ at $x = 0$ is characterized as follows, 
    \begin{align*}
        W^{(q)}(0) &= \begin{cases}
            0, & \textrm{if $X$ is of unbounded variation,}\\
            \delta^{-1}, & \textrm{if $X$ is of bounded variation,}
        \end{cases}
    \end{align*}
    where $\delta > 0$ is defined in \eqref{Eq: drift finite variation}.
\end{enumerate}
\end{remark}
We define, for $x \in \mathbb{R}$, 
\begin{align*}
    \overline{W}^{(q)}(x) \coloneqq \int_0^{x} W^{(q)}(y)\, \diff y \qquad \textrm{and} \qquad
    \overline{\overline{W}}^{(q)}(x) \coloneqq \int_0^{x} \overline{W}^{(q)}(y)\, \diff y.
\end{align*}
For $x \in \mathbb{R}$ and $r > 0$, we also define the \emph{second scale function},
\begin{align}\label{Eq: scale function Z}
    Z^{(q)}(x, \Phi_{q + r}) &\coloneqq e^{\Phi_{q + r}x} \left(1 - r\int_0^{x} e^{-\Phi_{q + r}z} W^{(q)}(z)\, \diff z \right) = r\int_0^{\infty} e^{-\Phi_{q + r}z} W^{(q)}(z + x) \,\diff z
\end{align}
where the second equality holds because \eqref{Eq: scale function laplace} gives $\int_0^\infty  \mathrm{e}^{-\Phi_{q+r} x} W^{(q)}(x) \diff x = r^{-1}$.
Differentiating the second scale function with respect to the first argument, we obtain
\begin{align}
    Z^{(q)'}(x,\Phi_{q+r}) &\coloneqq \frac{\partial}{\partial x} Z^{(q)}(x,\Phi_{q+r}) = \Phi_{q+r} Z^{(q)}(x,\Phi_{q+r}) - rW^{(q)}(x), \quad x > 0. \label{Eq: Z' phi}
\end{align}
Moreover, for $q\geq 0$ and $x \in \R$, we define 
\begin{align*}
    Z^{(q)}(x) &\coloneqq 1 + q\overline{W}^{(q)}(x), \\
    \overline{Z}^{(q)}(x) &\coloneqq \int^x_0 Z^{(q)}(y)\, \diff y = x + q \overline{\overline{W}}^{(q)}(x).
\end{align*}
We define the following for $b > a$, $x \in \mathbb{R}$, and $h: \mathbb{R} \to \mathbb{R}$: 
\begin{align}
    \rho_{a, b}^{(q)}(x; h) &\coloneqq \int^{b}_a W^{(q)}(x - y)h(y)\, \diff y,  \label{Eq: rho} \\
    \rho_{a, b}^{(q, r)}(x; h) &\coloneqq \rho_{a, b}^{(q)}(x; h) + r\int^{x}_b W^{(q + r)}(x - y)\rho_{a, b}^{(q)}(y; h)\, \diff y. \label{Eq: rho_r}
\end{align}
Additionally, we define
\begin{align}
    W^{(q, r)}_{a, b}(x) &\coloneqq W^{(q)}(x - a) + r\int^x_b W^{(q + r)}(x - y)W^{(q)}(y - a)\, \diff y, \label{Eq: W_b, a} \\
    Z^{(q, r)}_{a, b}(x) &\coloneqq Z^{(q)}(x - a) + r\int^x_bW^{(q + r)}(x - y) Z^{(q)}(y - a)\, \diff y, \label{Eq: Z_a, b} \\
    \overline{Z}^{(q, r)}_{a, b}(x) &\coloneqq \overline{Z}^{(q)}(x - a) + r\int^x_bW^{(q + r)}(x - y) \overline{Z}^{(q)}(y - a)\, \diff y. \label{Eq: bar Z_a, b}
\end{align}
For the periodic-classical barrier strategies, we decompose the NPV of costs as
\begin{equation*}
    v_{a,b}(x) = v_{a, b}^{LR}(x) + v_{a, b}^{f}(x),
\end{equation*}
where
\begin{align*}
        v_{a, b}^{LR}(x) &\coloneqq C_D\mathbb{E}_{x}\left[ \int^\infty_0 e^{-qt}\, \diff L^{a, b}(t)\right] + C_U\mathbb{E}_{x}\left[ \int^\infty_0 e^{-qt}\, \diff  R^{a, b}(t)\right], \\
        v_{a, b}^{f}(x) &\coloneqq \mathbb{E}_x\left[\int^\infty_0e^{-qt} f(Y^{a, b}(t))\, \diff  t\right].
\end{align*}

Using \eqref{Eq: W_b, a}--\eqref{Eq: bar Z_a, b}, an explicit expression for the control costs was derived in the proof of Lemma 3.1 in \cite{noba_optimal_2018}. Generalizing their results, we obtain the following lemma. 
\begin{lemma}[Lemma 3.1 in \cite{noba_optimal_2018}]\label{Lemma: control cost}
    For $a < b$ and $x \in \mathbb{R}$,
\begin{multline}\label{Eq: control cost}
    v_{a, b}^{LR}(x) = \left(\frac{r}{q\Phi_{q + r}} \frac{C_UZ^{(q)}(b - a) + C_D}{Z^{(q)}(b - a, \Phi_{q + r})} + \frac{C_U}{\Phi_{q + r}}\right) \left(Z^{(q, r)}_{a, b}(x) - rZ^{(q)}(b - a)\overline{W}^{(q + r)}(x - b)\right)\\
    - rC_D\overline{\overline{W}}^{(q + r)}(x - b) - C_U\left(\overline{Z}^{(q, r)}_{a, b}(x) + \frac{\psi'(0+)}{q} - r\overline{Z}^{(q)}(b - a) \overline{W}^{(q + r)}(x - b)\right).
\end{multline}
\end{lemma}

We provide an explicit expression for the running cost, which is computed using the proof technique from Proposition 5.1 in \cite{mata_bailout_2023}. The proof is given in Appendix \ref{proof_Lemma: inventory cost}.
\begin{lemma}\label{Lemma: inventory cost}
    For $a < b$ and $x \in \mathbb{R}$,
    \begin{multline}\label{Eq: inventory cost}
        v_{a, b}^{f}(x) = \frac{1}{q}\left(f(a) + \int^{\infty}_a f'(y)\frac{Z^{(q)}(b - y, \Phi_{q + r})}{Z^{(q)}(b - a, \Phi_{q + r})}\, \diff y\right) Z^{(q, r)}_{a, b}(x) - \rho_{a, b}^{(q, r)}(x; f)\\
        -r\overline{W}^{(q + r)}(x - b)v_{a, b}^{f}(b) - \int^{x}_b f(y) W^{(q + r)}(x - y)\, \diff y,
    \end{multline}
    where 
    \begin{align*}
    \begin{split}
        v_{a, b}^{f}(b) = -\rho_{a, b}^{(q)}(b; f) + \frac{1}{q}\left(f(a) + \int^{\infty}_a f'(y)\frac{Z^{(q)}(b - y, \Phi_{q + r})}{Z^{(q)}(b - a, \Phi_{q + r})}\, \diff y\right)Z^{(q)}(b - a). 
    \end{split}
    \end{align*}
\end{lemma}

Combining the above results and via integration by parts, we obtain the following, with the proof deferred to Appendix \ref{Sect: proof of NPV of costs}. Recall $\tilde{f}$ as defined in \eqref{f_tilde_def}.
\begin{lemma}\label{Lemma: NPV of costs}
    For $a < b$ and $x \in \mathbb{R}$, 
\begin{multline}\label{Eq: NPV of costs}
    v_{a,b}(x) = \frac{1}{q}\left(\tilde{f}(a) + \frac{\Gamma(a, b)}{Z^{(q)}(b - a, \Phi_{q + r})}\right)\left(Z^{(q, r)}_{a, b}(x) -r\overline{W}^{(q + r)}(x - b)Z^{(q)}(b - a) \right)\\
    - r(C_U + C_D)\overline{\overline{W}}^{(q + r)}(x - b) - C_Ux + r\overline{W}^{(q + r)}(x - b)\rho_{a, b}^{(q)}(b; \tilde{f}) - \tilde{f}(b)\overline{W}^{(q + r)}(x - b)\\
    - \rho_{a, b}^{(q, r)}(x; \tilde{f}) - C_U\frac{\psi'(0+)}{q} - \int^{x}_b \overline{W}^{(q + r)}(x - y)\tilde{f}'(y)\, \diff y,
\end{multline}
    where
    \begin{align}\label{Eq: Gamma function}
        \Gamma(a, b) &\coloneqq \int^{\infty}_a \tilde{f}'(y) Z^{(q)}(b - y, \Phi_{q + r})\, \diff y + \frac{r}{\Phi_{q + r}} (C_U + C_D), \quad a < b.
    \end{align}
\end{lemma}

\subsection{Smoothness of $v_{a, b}$} \label{Sect: smoothness}
In this section, we analyze the smoothness of the NPV of costs as defined in \eqref{Eq: NPV of costs}. Define the partial derivative of $\Gamma$, as given in \eqref{Eq: Gamma function}, as follows:
\begin{align}\label{Eq: gamma function}
\begin{split}
    \gamma(a, b) \coloneqq \frac{\partial}{\partial b} \Gamma(a, b) &= \Phi_{q+r} \int^{\infty}_a \tilde{f}'(y) Z^{(q)}(b - y, \Phi_{q + r})\, \diff y - r\rho_{a, b}^{(q)}(b; \tilde{f}') \\
    &= \Phi_{q + r} \Gamma(a, b) - r\left(C_U + C_D + \rho_{a, b}^{(q)}(b; \tilde{f}')\right), \quad a < b,
    \end{split}
\end{align}
where the first equality holds by \eqref{Eq: Z' phi}. It is worth noting that $\gamma(a, b)$ is independent of the values of $C_U$ and $C_D$. The last equality in \eqref{Eq: gamma function} is included solely to simplify complex expressions that appear later in the paper.
\begin{lemma}\label{Lemma: smoothness lemma}
    For $a < b$ and $x \in \mathbb{R}\backslash \{a\}$, 
    \begin{multline}\label{Eq: value derivative}
    \begin{split}
        v'_{a, b}(x) = \frac{\Gamma(a, b)}{Z^{(q)}(b - a, \Phi_{q + r})} W^{(q, r)}_{a, b}(x) - \rho_{a, b}^{(q, r)}(x; \tilde{f}') - \int^{x}_{b} W^{(q + r)}(x - y) \tilde{f}'(y)\, \diff y\\
        - r(C_U + C_D) \overline{W}^{(q + r)}(x - b) - C_U.
    \end{split}
    \end{multline}
    If $X$ is of unbounded variation,
    \begin{align}
    \begin{split}\label{Eq: value 2nd derivative}
        v''_{a, b}(x) &= \frac{\Gamma(a, b)}{Z^{(q)}(b - a, \Phi_{q + r})} \left(\frac{\diff }{\diff x} W^{(q, r)}_{a, b}(x)\right)\\
        &- \left(\int^{b}_a W^{(q)'}(x - y)\tilde{f}'(y)\, \diff y + r\int^{x}_b W^{(q + r)'}(x - y)\rho_{a, b}^{(q)}(y; \tilde{f}')\, \diff y\right)\\
        &- \int^{x}_{b} W^{(q + r)'}(x - y) \tilde{f}'(y)\, \diff y - r(C_U + C_D) W^{(q + r)}(x - b), \quad x \neq a.
    \end{split}
    \end{align}
\end{lemma}

The proof of Lemma \ref{Lemma: smoothness lemma} is deferred to Appendix \ref{Sect: proof of smoothness lemma}. In \eqref{Eq: value 2nd derivative}, the derivative $\frac{\diff }{\diff x} W^{(q, r)}_{a, b}(x)$ is well-defined for $x \neq a$ in light of \eqref{Eq: W_b, a} and the smoothness of scale functions as in Remark \ref{remark_smoothness_scale_function}.

In the remainder of this paper, we call a function \emph{sufficiently smooth} if it is continuously differentiable on $\mathbb{R}$ when $X$ is of bounded variation and twice continuously differentiable on $\mathbb{R}$ when $X$ is of unbounded variation.

\begin{lemma}\label{Lemma: smoothness remark}
    For $a < b$, the function $x \mapsto v_{a, b}(x)$ is sufficiently smooth if and only if $\Gamma(a,b) = 0$.
\end{lemma}
\begin{proof}
    By Lemma \ref{Lemma: smoothness lemma}, we have
    \begin{equation*}
    \begin{split}
        v'_{a, b}(a-) = -C_U, \quad v'_{a, b}(a+) = -C_U + \frac{\Gamma(a, b)}{Z^{(q)}(b - a, \Phi_{q + r})} W^{(q)}(0),
    \end{split}
    \end{equation*}
    in the case that $X$ is of bounded variation, and 
    \begin{equation*}
    \begin{split}
        v''_{a, b}(a-) = 0, \quad v''_{a, b}(a+) = \frac{\Gamma(a, b)}{Z^{(q)}(b - a, \Phi_{q + r})} W^{(q)'}(0+),
    \end{split}
    \end{equation*}
    when $X$ is of unbounded variation. The lemma now follows using Remark \ref{remark_smoothness_scale_function}(3) and $W^{(q)'}(0+) \neq 0$ as in Lemma 3.2 of \cite{kuznetsov_theory_2013}.
\end{proof}

\subsection{Candidate barriers} \label{subsection_candidate_barriers}
Previous works, such as \cite{baurdoux_optimality_2015} and \cite{perez_optimal_2020}, have shown the optimality of barrier-type strategies for similar stochastic control problems by selecting an optimal barrier, or an optimal pair of barriers, according to the smooth-fit principle. For example, in the continuous-monitoring version of this problem studied in \cite{baurdoux_optimality_2015}, an optimal pair of barriers, denoted as $(a^*, b^*)$, is chosen such that the corresponding NPV of costs is sufficiently smooth. In this paper, we instead seek $(a^*, b^*)$ such that 
\begin{align*}
    v_{a^*, b^*}^{f'}(a^*) &= \mathbb{E}_{a^*}\left[\int^\infty_0e^{-qt} f'(Y^{a^*, b^*}(t))\, \diff  t\right] = -C_U,\\
    v_{a^*, b^*}^{f'}(b^*) &= \mathbb{E}_{b^*}\left[\int^\infty_0e^{-qt} f'(Y^{a^*, b^*}(t))\, \diff  t\right] = C_D. 
\end{align*}
This approach is recently recognized in \cite{noba_stochastic_2022} as a potential alternative for solving the kind of control problems that is currently studied. It is noted that the optimality of the single-barrier strategy for various one-sided control problems can be shown using this method within a one-sided control setting. By reviewing the two-sided control problem in \cite{baurdoux_optimality_2015} through this approach, we can further confirm that the conditions outlined above are satisfied. Hence, we expect that this approach will be similarly effective in identifying optimal barriers for our current problem.

The advantage of this approach is that it allows for a probabilistic interpretation. We observe that, intuitively, $\mathbb{E}_{x}\left[\int^\infty_0e^{-qt} f'(Y^\pi(t))\, \diff  t\right]$ represents the rate of change of the running cost when the controlled process $Y^\pi$ with initial value $x$ is shifted to the right by an infinitesimal amount. Thus, it is desirable for a strategy $\pi$ to keep $Y^{\pi}$ within the set $\mathcal{C}^\pi$ for as long as possible, where 
\[\mathcal{C}^\pi = \left\{x \in \mathbb{R}: -C_U \leq \mathbb{E}_{x}\left[\int^\infty_0e^{-qt} f'(Y^{\pi}(t))\, \diff  t\right] \leq C_D \right\}.\]
Furthermore, for the current problem, this approach is shown to be equivalent to the smooth-fit approach. We first establish this equivalence, where the following result proves useful. This result can again be derived by the resolvent measure obtained in Proposition 5.1 in \cite{mata_bailout_2023}, with the proof deferred to Appendix \ref{Sect: proof of auxiliary lemma}.
\begin{lemma}\label{Lemma: auxiliary result}
    For $a < b$ and $x \in \mathbb{R}$, we have
    \begin{multline}\label{Eq: convexity misc 1}
        v^{f'}_{a, b}(x) = \frac{\gamma(a, b)}{qZ^{(q)}(b - a; \Phi_{q + r})} Z^{(q, r)}_{a, b}(x) - \rho_{a, b}^{(q, r)}(x; \tilde{f}') - \int^{x}_b W^{(q + r)}(x - y)\tilde{f}'(y)\, \diff y\\
        - r\overline{W}^{(q + r)}(x - b)\left(-\rho_{a, b}^{(q)}(b; \tilde{f}') + \frac{Z^{(q)}(b - a)}{qZ^{(q)}(b - a; \Phi_{q + r})}\gamma(a, b)\right) - C_U.
    \end{multline}
In particular, at $x = a$ and $x = b$,
    \begin{align}
        v^{f'}_{a, b}(a) &= \frac{\gamma(a, b)}{qZ^{(q)}(b - a; \Phi_{q + r})} - C_U, \label{Eq: convexity misc 3} \\
        v^{f'}_{a, b}(b) &= \frac{Z^{(q)}(b - a)}{qZ^{(q)}(b - a; \Phi_{q + r})}\gamma(a, b) - \rho_{a, b}^{(q)}(b; \tilde{f}') - C_U. \label{Eq: convexity misc 2}
    \end{align}
\end{lemma}

\begin{proposition} \label{Prop: iff conditions} For $a < b$, the followings are equivalent:  
\begin{itemize}
    \item[(i)] $v^{f'}_{a, b}(a) = -C_U$ and $v^{f'}_{a, b}(b) = C_D$,
    \item[(ii)] $\Gamma(a, b) = 0$ and $\gamma(a,b) = 0$,
    \item[(iii)] $v_{a, b}$ is sufficiently smooth, with $v'_{a, b}(a) = -C_U$ and $v'_{a, b}(b) = C_D$.
\end{itemize}
\end{proposition}

\begin{proof}
Proof of (i) $\Longleftrightarrow$ (ii): First, by \eqref{Eq: convexity misc 3},  
\begin{align}
    v^{f'}_{a, b}(a) = -C_U \Longleftrightarrow \gamma(a, b) = 0. \label{Eq: equivalence 1}
\end{align}
By \eqref{Eq: gamma function}, we have
\begin{align}
\rho_{a, b}^{(q)}(b; \tilde{f}')
    &= \frac {\Phi_{q+r}  \Gamma(a, b) -  \gamma(a, b) } r - (C_U + C_D) . \label{rho_a_b_expression}
\end{align}
Substituting \eqref{rho_a_b_expression} into \eqref{Eq: convexity misc 2}, we have
\begin{align*}
    v^{f'}_{a, b}(b) &= -\frac{\Phi_{q+r}}{r} \Gamma(a, b) + \left(\frac{Z^{(q)}(b - a)}{qZ^{(q)}(b - a; \Phi_{q + r})} + \frac{1}{r}\right)\gamma(a, b) + C_D.
\end{align*}
The equivalence follows from the last equality and \eqref{Eq: equivalence 1}.

Proof of (ii) $\Longleftrightarrow$ (iii): Substituting \eqref{rho_a_b_expression} in  \eqref{Eq: value derivative} with $x = b$,
\begin{align*}
    v'_{a, b}(b) &= \left(\frac{W^{(q)}(b - a)}{Z^{(q)}(b - a, \Phi_{q + r})} - \frac{\Phi_{q+r}}{r}\right) \Gamma(a, b) + \frac{\gamma(a, b)}{r} + C_D.
\end{align*}
Thus, $v'_{a, b}(b) = C_D$ is equivalent to $\gamma(a, b) = 0$ when $\Gamma(a, b) = 0$. Finally, Lemma \ref{Lemma: smoothness remark} shows the equivalence.
\end{proof}

\subsection{Existence of candidate barriers}\label{Sect: existence of (a, b)}
We now establish the existence of candidate barriers $(a^*, b^*)$ satisfying 
\begin{align*}
    \mathfrak{C}: v^{f'}_{a^*, b^*}(a^*) = -C_U \quad \textrm{and} \quad v^{f'}_{a^*, b^*}(b^*) = C_D, 
\end{align*}
which is, by Proposition \ref{Prop: iff conditions}, equivalent to
\begin{align*}
    \mathfrak{C}': \Gamma(a^*,b^*) = \gamma(a^*, b^*) = 0.
\end{align*}

Recall $\bar{a} = \inf\{a \in \mathbb{R}: \tilde{f}'(a) \geq 0\}$ as defined in Assumption \ref{asm: on f}(2). We can immediately eliminate the half-line $[\bar{a}, \infty)$ from consideration for $a^*$. Indeed, for $a \geq \overline{a}$, we have $\tilde{f}'(y + a) \geq \tilde{f}'(y + \bar{a}) \geq 0$ for all $y \geq 0$ by Assumption \ref{asm: on f}(1). Thus, from \eqref{Eq: Gamma function}, we have $\Gamma(a, b) > 0$ for $b > a \geq \bar{a}$. Below, we show that, by decreasing the value of $a$ from $\bar{a}$, we reach $a^*$ such that the function $b \mapsto \Gamma(a^*, b)$ starts at $\Gamma(a^*, a^*+) > 0$ and becomes tangent to the x-axis at $b^*$, where its partial derivative $\gamma(a^*, b^*)$ is zero.

We first obtain the start and end values of the mapping $b \mapsto \Gamma(a, b)$.
\begin{lemma} \label{Lemma: F asymptotics} For all $a \in \mathbb{R}$,
\begin{align}
    \Gamma_1(a) &\coloneqq \Gamma(a, a+) = \int^\infty_0 e^{-\Phi_{q + r} y} \tilde{f}'(y+a) \, \diff y + \frac{r}{\Phi_{q+r}}(C_U + C_D), \label{Eq: F(a, a)} \\
    \Gamma_2(a) &\coloneqq \lim_{b \to \infty} \frac{\Gamma(a, b)}{Z^{(q)}(b - a, \Phi_{q + r})} = \int^{\infty}_0 e^{-\Phi_q y} \tilde{f}'(y + a)\, \diff y. \label{Eq: F(a, inf)}
\end{align}
\end{lemma}

The proof of Lemma \ref{Lemma: F asymptotics} is deferred to Appendix \ref{Sect: proof of F asymptotics}. Notice that both $a\mapsto \Gamma_1(a)$ and $a\mapsto \Gamma_2(a)$ are continuous. Moreover, both are strictly increasing on $(-\infty, \bar{a})$ due to the convexity of $\tilde{f}$ and Assumption \ref{asm: on f}(3). Define their left inverses,
\begin{align}\label{Eq: a_underline}
\begin{split}
    \underline{a}_i \coloneqq \inf\{a \in \mathbb{R}: \Gamma_i(a) \geq 0\}, \quad i = 1,2.
\end{split}
\end{align}

Observe that 
\[
\underline{a}_1 \in [-\infty, \bar{a}) \quad \textrm{and} \quad \underline{a}_2 \in (-\infty, \bar{a}).
\]
To establish that $\underline{a}_2 > -\infty$, note that by monotone convergence and the convexity of $\tilde{f}$, we have $\Gamma_2(a) \xrightarrow{a \downarrow -\infty}\int^{\infty}_0 e^{-\Phi_q z} \tilde{f}'(-\infty)\, \diff z \in [-\infty, 0)$. Hence, $\underline{a}_2 > -\infty$. Next, to show that both $\underline{a}_1$ and $\underline{a}_2$ are upper bounded by $\bar{a}$, note that by \eqref{Eq: sum C} and Assumption \ref{asm: on f}(2), we have 
\[ \Gamma_1(\bar{a}) = \int^\infty_0 e^{-\Phi_{q + r} y} \tilde{f}'(y + \bar{a})\, \diff y + \frac{r}{\Phi_{q + r}}(C_U + C_D) > 0,\]
which implies $\bar{a} > \underline{a}_1$. Additionally, by Assumption \ref{asm: on f}(3), $\Gamma_2(\bar{a}) > 0$, which implies $\bar{a} > \underline{a}_2$.

\begin{remark} \label{Remark: a_2} The value $\underline{a}_2$ is the optimal barrier in the case with only upward control; see Section 7 of \cite{yamazaki_inventory_2017}. By the definition of $\underline{a}_2$ in \eqref{Eq: a_underline}, we have
    \begin{equation*}
        0 = \frac{\Phi_{q}}{q} \Gamma_2(\underline{a}_2) = \frac{\Phi_{q}}{q} \int^\infty_0 e^{-\Phi_{q}y} f'(y + \underline{a}_2)\, \diff y + C_U = \mathbb{E}_{\underline{a}_2}\left[\int^\infty_0e^{-qt} f'(Y^{\underline{a}_2, \infty}(t))\, \diff t\right] + C_U,
    \end{equation*}
    where the last equality holds by Theorem 2.8(iii) in \cite{kuznetsov_theory_2013}. Thus,
\begin{align}
    v^{f'}_{\underline{a}_2, \infty}(\underline{a}_2) \coloneqq \mathbb{E}_{\underline{a}_2}\left[\int^\infty_0 e^{-qt} f'(Y^{\underline{a}_2, \infty}(t))\, \diff t\right] = -C_U.  \label{v_f_prime_at_a_2}
\end{align}
\end{remark}

Next, we show two auxiliary lemmas. We define
\[
\underline{\Gamma}(a) \coloneqq \inf_{b > a} \Gamma(a, b), \quad a \in \mathbb{R}.
\] 
\begin{lemma} \label{Lemma: F min}
    We have (i) $\underline{\Gamma}(\bar{a}) > 0$ and (ii) $\underline{\Gamma}(\underline{a}_1 \vee \underline{a}_2) < 0$.
\end{lemma}

\begin{proof}
(i) By \eqref{Eq: F(a, a)} and the definition of $\overline{a}$, we have $\Gamma(\bar{a}, \bar{a}+) > 0$. From \eqref{Eq: Gamma function}, it is easy to see that $b \mapsto \Gamma(\bar{a}, b)$ is strictly increasing on $(\bar{a}, \infty)$. Hence, $\underline{\Gamma}(\bar{a}) > 0$. 

\noindent (ii) We show separately for the cases $\underline{a}_2 \geq \underline{a}_1$ and $\underline{a}_2 < \underline{a}_1$. 

(1) Suppose $\underline{a}_2 \geq \underline{a}_1$. By \eqref{Eq: convexity misc 3} with $a = \underline{a}_2$ and \eqref{v_f_prime_at_a_2}, we have, for any $b > \underline{a}_2$, 
    \begin{equation*}
        \frac{\gamma(\underline{a}_2, b)}{qZ^{(q)}(b - \underline{a}_2; \Phi_{q + r})} - C_U = \mathbb{E}_{\underline{a}_2}\left[\int^\infty_0e^{-qt} f'(Y^{\underline{a}_2, b}(t))\, \diff t\right] < \mathbb{E}_{\underline{a}_2}\left[\int^\infty_0e^{-qt} f'(Y^{\underline{a}_2, \infty}(t))\, \diff t\right] = -C_U, 
    \end{equation*}
    where the inequality holds because $Y^{\underline{a}_2, b}(t) \leq Y^{\underline{a}_2, \infty}(t)$ for all $t \geq 0$. To see why the inequality is strict, first note that $\underline{a}_2 < \bar{a} < \bar{\bar{a}}$ implies the existence of some $c > \underline{a}_2$ and $\varepsilon > 0$ such that $f'(c) < f'(c + \varepsilon)$. Moreover, since the set $\{t \in [0,\infty): Y^{\underline{a}_2, b}(t) < c, Y^{\underline{a}_2, \infty}(t) > c + \varepsilon\}$ has a positive Lebesgue measure with positive probability, the inequality is indeed strict. This inequality implies 
    \begin{align}\label{Eq: iff misc 2}
        0 > \gamma(\underline{a}_2, b) = \Phi_{q+r} \Gamma(\underline{a}_2, b) - r\left(C_U + C_D + \rho_{\underline{a}_2, b}^{(q)}(b; \tilde{f}')\right), \quad b > \underline{a}_2.
    \end{align}
    Now, since $f'$ is non-decreasing and $f'(\infty) > qC_D$ by Assumption \ref{asm: on f}(3), for some sufficiently large $\tilde{b}$, 
    \begin{equation*}
        C_D < \mathbb{E}_{\tilde{b}} \left[\int^\infty_0 e^{-qt} f'(Y^{\underline{a}_2, \tilde{b}}(t))\, \diff t\right] = v^{f'}_{\underline{a}_2, \tilde{b}}(\tilde{b}) = \frac{Z^{(q)}(\tilde{b} - \underline{a}_2)}{qZ^{(q)}(\tilde{b} - \underline{a}_2; \Phi_{q + r})}\gamma(\underline{a}_2, \tilde{b}) - \rho_{\underline{a}_2, \tilde{b}}^{(q)}(\tilde{b}; \tilde{f}') - C_U,
    \end{equation*}
    where we used \eqref{Eq: convexity misc 2} for the last equality. This, and \eqref{Eq: iff misc 2} with $b = \tilde{b}$, gives $\rho_{\underline{a}_2, \tilde{b}}^{(q)}(\tilde{b}; \tilde{f}') + C_U + C_D < 0$. Substituting this inequality back into \eqref{Eq: iff misc 2}, we have
    $\Gamma(\underline{a}_2, \tilde{b}) < 0$, and therefore $\underline{\Gamma}(\underline{a}_1 \vee \underline{a}_2)  = \underline{\Gamma}(\underline{a}_2) \leq \Gamma(\underline{a}_2, \tilde{b}) < 0$.\\
    (2) Suppose $\underline{a}_1 >  \underline{a}_2$. By \eqref{Eq: gamma function} and the definition of $\underline{a}_1$, we have $\gamma(\underline{a}_1, \underline{a}_1+) = - r(C_U + C_D) < 0$. Since $\gamma(\underline{a}_1, \cdot)$ is the partial derivative of $\Gamma(\underline{a}_1, \cdot)$, which starts at $\Gamma(\underline{a}_1,\underline{a}_1+) = 0$, the function $\Gamma(\underline{a}_1, \cdot)$ must go below zero. Thus, $0 > \underline{\Gamma}(\underline{a}_1) = \underline{\Gamma}(\underline{a}_1 \vee \underline{a}_2)$.
\end{proof}

\begin{lemma} \label{Lemma: monotonicity F} The mapping $a \mapsto \underline{\Gamma}(a)$ is continuous and strictly increasing on $(\underline{a}_1 \vee \underline{a}_2, \bar{a})$.
\end{lemma}
\begin{proof}
By taking the partial of $\Gamma$ with respect to $a$, we obtain, for a.e.\ $a$ at which $\tilde{f}'(a)$ exists,
\begin{equation*}
    \frac{\partial }{\partial a}\Gamma(a, b) = - \tilde{f}'(a)Z^{(q)}(b - a, \Phi_{q + r}), \quad a < b, 
\end{equation*}
which is positive for $a < \bar{a}$. 
For $a' < a \leq \bar{a}$, notice that
\begin{align*}
    \underline{\Gamma}(a') &\leq \inf_{b > a} \Gamma(a', b) = \inf_{b > a} \left(\Gamma(a, b) + \int^a_{a'} \tilde{f}'(y) Z^{(q)}(b - y, \Phi_{q + r})\, \diff y\right)\\
    &\leq \inf_{b > a} \left(\Gamma(a, b) + \int^{(a + a')/2}_{a'} \tilde{f}'(y) Z^{(q)}(b - y, \Phi_{q + r})\, \diff y\right)\\
    &\leq \inf_{b > a} \left(\Gamma(a, b) + Z^{(q)}\left(b - \frac{a + a'}{2}, \Phi_{q + r}\right)\int^{(a + a')/2}_{a'} \tilde{f}'(y)\, \diff y\right)\\
    &\leq \underline{\Gamma}(a) + Z^{(q)}\left(\frac{a - a'}{2}, \Phi_{q + r}\right)\int^{(a + a')/2}_{a'} \tilde{f}'(y)\, \diff y < \underline{\Gamma}(a).
\end{align*}
Hence, $a \mapsto \underline{\Gamma}(a)$ is strictly increasing on $(-\infty, \bar{a})$. Continuity is guaranteed because $\Gamma(a, b) \xrightarrow{b \uparrow \infty} \infty$ for all $a > \underline{a}_2$, by \eqref{Eq: F(a, inf)} and the definition of $\underline{a}_2$ as in \eqref{Eq: a_underline}.
\end{proof}

We are now ready to show the existence and uniqueness of the candidate barriers.
\begin{proposition}\label{Prop: existence prop}
    There exists a unique pair $(a^*, b^*)$ such that $\mathfrak{C}$, or equivalently, $\mathfrak{C}'$, holds.
\end{proposition}
\begin{proof}
By Lemmas \ref{Lemma: F min} and \ref{Lemma: monotonicity F}, there exists a unique root  $a^* \in (\underline{a}_1 \vee \underline{a}_2, \infty)$ such that $\underline{\Gamma}(a^*) = 0$. Moreover, because $\Gamma(a^*, a^*+) > 0$ and $\lim_{b \to \infty}\Gamma(a^*,b) = \infty$ in view of the definitions of $\underline{a}_1$ and $\underline{a}_2$, the minimum of $b \mapsto \Gamma(a^*,b)$ is attained at some $b^* \in (a^*, \infty)$. By the continuity of $b \mapsto \gamma(a^*,b)$, it must hold that $\gamma(a^*,b^*)=0$. Thus, $\mathfrak{C}'$, or equivalently, $\mathfrak{C}$, holds. Moreover, such a $b^*$ is unique. To establish uniqueness, recall that $b^*$ must satisfy $v^{f'}_{a^*, b^*}(b^*) = C_D$, as required by $\mathfrak{C}$. Given the probabilistic expression $v^{f'}_{a^*, b}(b) = \mathbb{E}_b \left[\int^\infty_0e^{-qt} f'(Y^{a^*, b}(t))\, \diff  t\right]$, it is clear that the mapping $b \mapsto v^{f'}_{a^*, b}(b)$ is non-decreasing on $(a^*, \infty)$. Further, by applying the argument used in the proof of Lemma \ref{Lemma: F min}, we see that $b \mapsto v^{f'}_{a^*, b}(b)$ is, in fact, strictly increasing. Hence, there is exactly one $b^* > a^*$ such that $\mathfrak{C}'$ (equivalently $\mathfrak{C}$) holds.
\end{proof}

\section{Verification}
In the remainder of this paper, we denote by $(a^*, b^*)$ the pair of barriers satisfying $\mathfrak{C}$ (equivalently $\mathfrak{C}'$), whose existence and uniqueness are shown in Proposition \ref{Prop: existence prop}. In this section, we prove the optimality of the double-barrier strategy $\pi_{a^*, b^*}$. To this end, we apply the conventional verification technique using It\^o's lemma, which has been employed in previous works such as \cite{avram_exit_2004, baurdoux_optimality_2015, perez_optimal_2020}, among others. 

Let $\mathcal{L}$ be the infinitesimal generator associated with the spectrally negative L\'evy process $X$ applied to a sufficiently smooth function $h: \mathbb{R} \to \mathbb{R}$, 
\begin{equation*}
    \mathcal{L}h(x) \coloneqq \gamma h'(x) + \frac{\sigma^2}{2}h''(x) + \int_{(-\infty, 0)} \left[h(x + z) - h(x) - h'(x) z 1_{\{-1 < z < 0\}} \,\mu(\diff z)\right], \quad x\in \mathbb{R}. 
\end{equation*}
Also, we define the operator $\mathcal{M}$ for a measurable function $h$, 
\begin{equation*}
    \mathcal{M}h(x) \coloneqq \inf_{l \geq 0} \{C_Dl + h(x - l)\},
\end{equation*}
which has been used to prove the optimality of strategies in stochastic control problems in works such as \cite{mata_bailout_2023, perez_optimal_2020}. The following lemma gives the conditions that are sufficient for the optimality of $\pi_{a^*, b^*}$; its proof is deferred to Appendix \ref{Sect: proof of verification lemma}.
\begin{lemma}[verification lemma]\label{Lemma: verification lemma}
Suppose $w: \mathbb{R} \to \mathbb{R}$ is the NPV of costs under an admissible strategy. If it is sufficiently smooth on $\mathbb{R}$, has polynomial growth, satisfies $w' \geq -C_U$, 
    \begin{equation*}
        (\mathcal{L}-q)w(x) + r(\mathcal{M}w(x) - w(x)) + f(x) \geq 0, \quad x\in \mathbb{R}, 
    \end{equation*}
    and $\limsup_{t, n \uparrow \infty} \mathbb{E}_x\left[e^{-q(t\wedge \tau_n)} w(Y^\pi(t\wedge \tau_n))\right] \leq 0$ for all admissible strategies $\pi \in \Pi$, where $\tau_n \coloneqq \inf\{t \geq 0: |Y^{\pi}(t)| > n\}$, then $w(x) = v(x) = \inf_{\pi \in \Pi} v^\pi(x)$.
\end{lemma}

In the remainder of this section, we show that the conditions in Lemma \ref{Lemma: verification lemma} are satisfied for the function $v_{a^*, b^*}$. First, by Lemma \ref{Lemma: smoothness remark}, $v_{a^*, b^*}$ is sufficiently smooth. Furthermore, by Assumption \ref{asm: on X} in conjunction with the linearity of $v_{a^*, b^*}$ below $a^*$, the integral component of $\mathcal{L}v_{a^*, b^*}$ is finite on $\mathbb{R}$. Thus, $\mathcal{L}v_{a^*, b^*}$ makes sense everywhere on $\mathbb{R}$. 

Next, we establish the related properties for the computation of $\mathcal{M} v_{a^*,b^*}$. Notice that, with the condition $\mathfrak{C}'$, the function $v_{a^*, b^*}$ is simplified. 
By $\mathfrak{C}'$ and \eqref{Eq: value derivative}, we have
\begin{equation} \label{v_star_derivative}
    v'_{a^*, b^*}(x) = - \rho_{a^*, b^*}^{(q, r)}(x; \tilde{f}') - \int^{x}_{b^*} W^{(q + r)}(x - y) \tilde{f}'(y)\, \diff y  - r(C_U + C_D) \overline{W}^{(q + r)}(x - b^*) - C_U.
\end{equation}

\begin{lemma}\label{Lemma: convexity}
    We have
\begin{enumerate}
    \item[(1)]  $v^{f'}_{a^*, b^*}(x) = v'_{a^*, b^*}(x)$ for all $x \in \mathbb{R}$, 
    \item[(2)] $x \mapsto v_{a^*, b^*}(x)$ is convex on $\mathbb{R}$,
    \item[(3)] $v'_{a^*, b^*}(x) \geq -C_U$ for all $x \in \mathbb{R}$.
\end{enumerate}
\end{lemma}

\begin{proof}
(1) By Proposition \ref{Prop: iff conditions}, we have $v'_{a^*, b^*}(b^*) = C_D$. Combining this with the expression for $v'_{a^*, b^*}$ in \eqref{v_star_derivative}, we get
    \begin{align}
        C_D = -\rho_{a^*, b^*}^{(q)}(b^*; \tilde{f}') - C_U.\label{Eq: convexity misc 4}
    \end{align}
Substituting $\gamma(a^*, b^*) = 0$ and \eqref{Eq: convexity misc 4} into \eqref{Eq: convexity misc 1}, $v^{f'}_{a^*, b^*}(x)$ matches  \eqref{v_star_derivative}, as required.
\\
(2) By the convexity of $f$ and the monotonicity of $Y^{a^*, b^*}$ in the starting point, the mapping $x \mapsto v^{f'}_{a^*, b^*}(x)$ is monotone. Hence, by (1), $v_{a^*, b^*}$ is convex.
\\
(3) By (2) and as $v'_{a^*, b^*}(x) = -C_U$ for $x \in (-\infty, a^*]$, $v'_{a^*, b^*}(x) \geq -C_U$ for all $x \in \mathbb{R}$. 
\end{proof}

As a direct consequence of Lemma \ref{Lemma: convexity}(ii) and the condition $\mathfrak{C}$, $\mathcal{M} v_{a^*,b^*}$ has the following characterization.
\begin{corollary} \label{corollary_generator} 
We have
\begin{align*}
    \mathcal{M}v_{a^*, b^*}(x) &= \begin{cases}
        v_{a^*, b^*}(x) , & x < b^*\\
        C_D(x - b^*) + v_{a^*, b^*}(b^*), & x \geq b^*.\\
    \end{cases}
\end{align*}
\end{corollary}

Furthermore, standard computations lead to the following characterization of $(\mathcal{L} - q) v_{a^*, b^*}$, with the proof deferred to Appendix \ref{Sect: proof of verification generator}. 
\begin{lemma} \label{Lemma: generator M}
We have
\begin{align*}
    (\mathcal{L} - q) v_{a^*, b^*}(x) + f(x) &= \begin{cases}
        \tilde{f}(x) - \tilde{f}(a^*), & x \leq a^*, \\
        0, & a^* < x < b^*,\\
        -r\left(v_{a^*, b^*}(b^*) - v_{a^*, b^*}(x) + C_D(x - b^*)\right), & x \geq b^*.\\
    \end{cases}
\end{align*}
\end{lemma}

The polynomial growth of $v_{a^*, b^*}$ can be established using the proof technique from Lemma 4.7 in \cite{perez_optimal_2020}. As the method is nearly identical, the proof of the following result is omitted.
\begin{lemma}
    The function $x \mapsto v_{a^*, b^*}(x)$ is of polynomial growth.
\end{lemma}

\begin{lemma}\label{Lemma: verification limit}
    For all admissible strategies $\pi \in \Pi$, 
    \[\limsup_{t, n \uparrow \infty} \mathbb{E}_x \left[e^{-q(t\wedge \tau_n)} v_{a^*, b^*}(Y^\pi(t\wedge \tau_n))\right] \leq 0.\]
\end{lemma}

The proof of Lemma \ref{Lemma: verification limit} is deferred to Appendix \ref{Sect: proof of verification limit}.
\begin{theorem}\label{Thm: optimal strategy}
The strategy $\pi_{a^*,b^*}$ is optimal with
    \[v_{a^*, b^*}(x) = \inf_{\pi \in \Pi} v^\pi(x), \quad x\in \mathbb{R}.\]
\end{theorem}
\begin{proof}
    Notice that $v_{a^*, b^*}$ is sufficiently smooth on $\mathbb{R}$, has polynomial growth, and satisfies $v_{a^*, b^*}'\geq -C_U$ by Lemma \ref{Lemma: convexity}. By Corollary \ref{corollary_generator} and Lemma \ref{Lemma: generator M}, (1) $(\mathcal{L}-q)v_{a^*, b^*}(x) + r(\mathcal{M}v_{a^*, b^*}(x) - v_{a^*, b^*}(x)) + f(x) = \tilde{f}(x)-\tilde{f}(a^*) \geq 0$ for $x \leq a^*$, where the inequality holds because $a^* < \bar{a}$ and by Assumption \ref{asm: on f}(2), (2) $(\mathcal{L}-q)v_{a^*, b^*}(x) + r(\mathcal{M}v_{a^*, b^*}(x) - v_{a^*, b^*}(x)) + f(x) = 0$ for $x > a^*$. Hence,
    \begin{equation*}
        (\mathcal{L}-q)v_{a^*, b^*}(x) + r(\mathcal{M}v_{a^*, b^*}(x) - v_{a^*, b^*}(x)) + f(x) \geq 0, \quad x\in \mathbb{R}. 
    \end{equation*}
    Since $v_{a^*,b^*}$ is the NPV of costs under the admissible strategy $\pi_{a^*, b^*}$, by applying Lemma 
    \ref{Lemma: verification lemma}, we have that $v_{a^*, b^*}(x) = v(x)$ for all $x \in \mathbb{R}$.
\end{proof}
\section{Numerical example} \label{Sect: numerical result}
In this section, we consider a specific problem and determine the optimal double-barrier strategy numerically. Let $f(x) = x^2$, and $X$ be a spectrally negative L\'evy process with exponential jumps given by
\begin{equation*}
    X(t) = x + t + B(t) + \sum^{N(t)}_{n = 1} Z_n,   
\end{equation*}
where $B = (B(t); t\geq 0)$ is a standard Brownian motion, $N = (N(t); t\geq 0)$ is a Poisson process with arrival rate $\lambda = 0.2$, and $\{Z_n\}_{n \geq 1}$ is a sequence of i.i.d. exponential random variables with rate $1$. By Example 3.1 of \cite{egami_phase-type_2014}, the $q$-scale function is
\begin{equation}\label{Eq: W numerical}
    W^{(q)}(x) = \frac{e^{\Phi_q x}}{\psi'(\Phi_q)} - \sum^{2}_{i = 1} B_{i, q} e^{-\xi_{i, q} x}, 
\end{equation}
where, for all $i =1, 2$, we have $B_{i, q} = -1/\psi'(-\xi_{i, q})$, with $\xi_{i, q}$ defined such that $\psi(-\xi_{i, q}) = q$ and $-\xi_{i, q} < 0$. Furthermore, we set $q = 0.05$, $C_U = 200$, and $C_D = 200$. The Poisson process modeling the downward control opportunities, denoted by $N^r$, has a rate of $r$. 

Figure \ref{Fig: numerical ex} shows the value function $v_{a^*, b^*}$ alongside the NPV of costs $v_{a, b}$ for various barrier strategies where $(a, b) \neq (a^*, b^*)$, with the rate parameter for $N^r$ (the Poisson process modeling the downward control opportunities) set to $r = 0.1$. It is clear that $v_{a^*, b^*}(x) \leq v_{a, b}(x)$ for all $x \in \mathbb{R}$ when $(a, b) \neq (a^*, b^*)$.

\begin{figure}[h]
\centering
\begin{subfigure}[b]{0.49\textwidth}
    \centering
    \includegraphics[width=\textwidth]{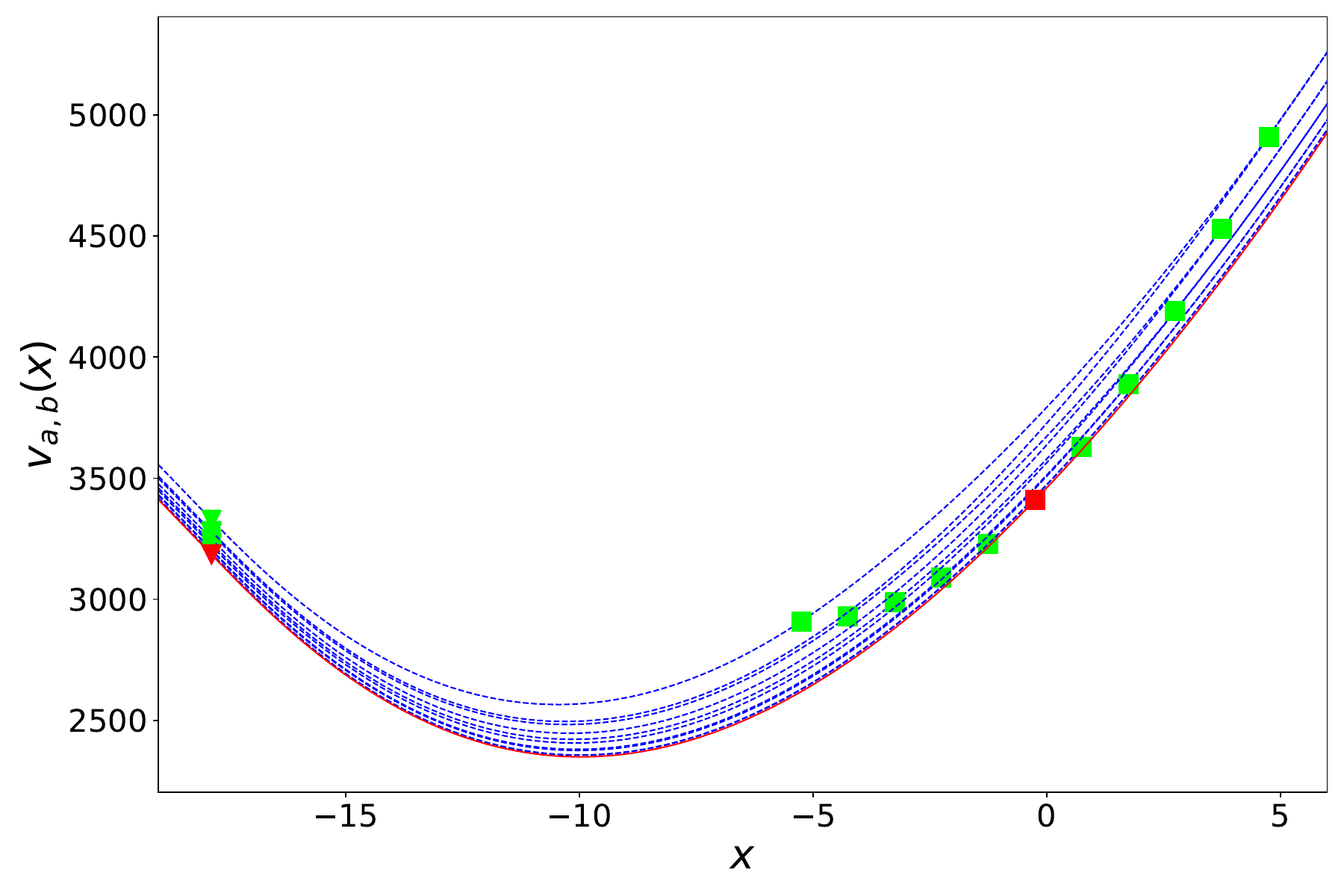}
\end{subfigure}
\vspace{1em}
\begin{subfigure}[b]{0.49\textwidth}
    \centering
    \includegraphics[width=\textwidth]{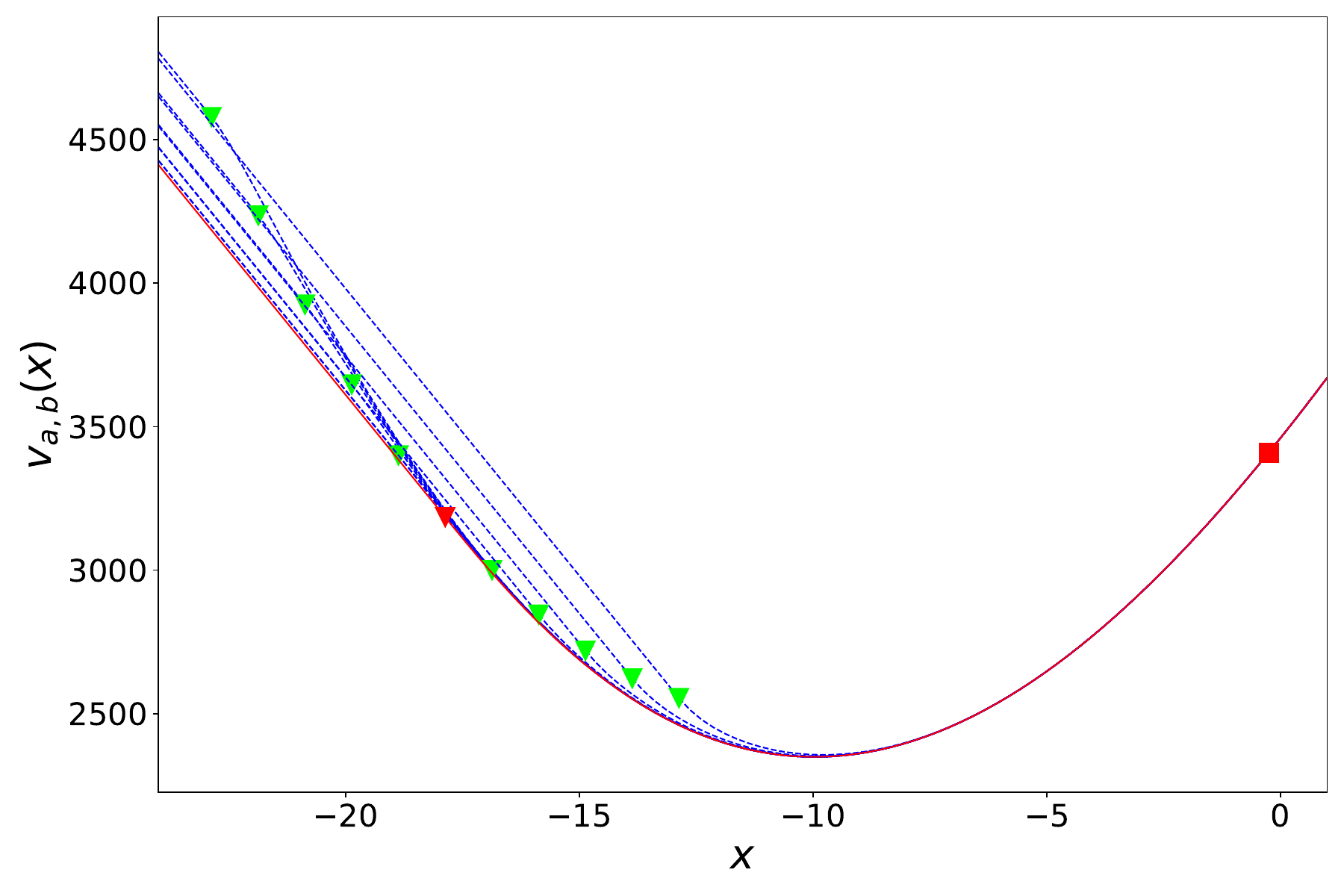}
\end{subfigure}

\caption{Plots of $v_{a, b}$ (for $r = 0.1$) against initial position $x$. Left: Plot of $v_{a^*, b}$ (blue dashed) for $b = b^* - 5, b^* - 4, \dots, b^* - 1, b^* + 1, \dots, b^* + 5$, with $(a^*, v_{a^*, b}(a^*))$ (lime triangle), $(a^*, v_{a^*, b^*}(a^*))$ (red triangle), $(b, v_{a^*, b}(b))$ (lime square), and $(b^*, v_{a^*, b^*}(b^*))$ (red square). Right: Plot of $v_{a, b^*}$ (blue dashed) for $a = a^* - 5, a^* - 4, \dots, a^* - 1, a^* + 1, \dots, a^* + 5$, with $(a, v_{a, b^*}(a))$ (lime triangle), $(a^*, v_{a^*, b^*}(a^*))$ (red triangle), $(b^*, v_{a, b^*}(b^*))$ (lime square), and $(b^*, v_{a^*, b^*}(b^*))$ (red square). In both plots, the value function $v_{a^*, b^*}$ is indicated by red curves.}
\label{Fig: numerical ex}
\end{figure}

\begin{figure}[h]
\centering
\begin{subfigure}[b]{0.49\textwidth}
    \centering
    \includegraphics[width=\textwidth]{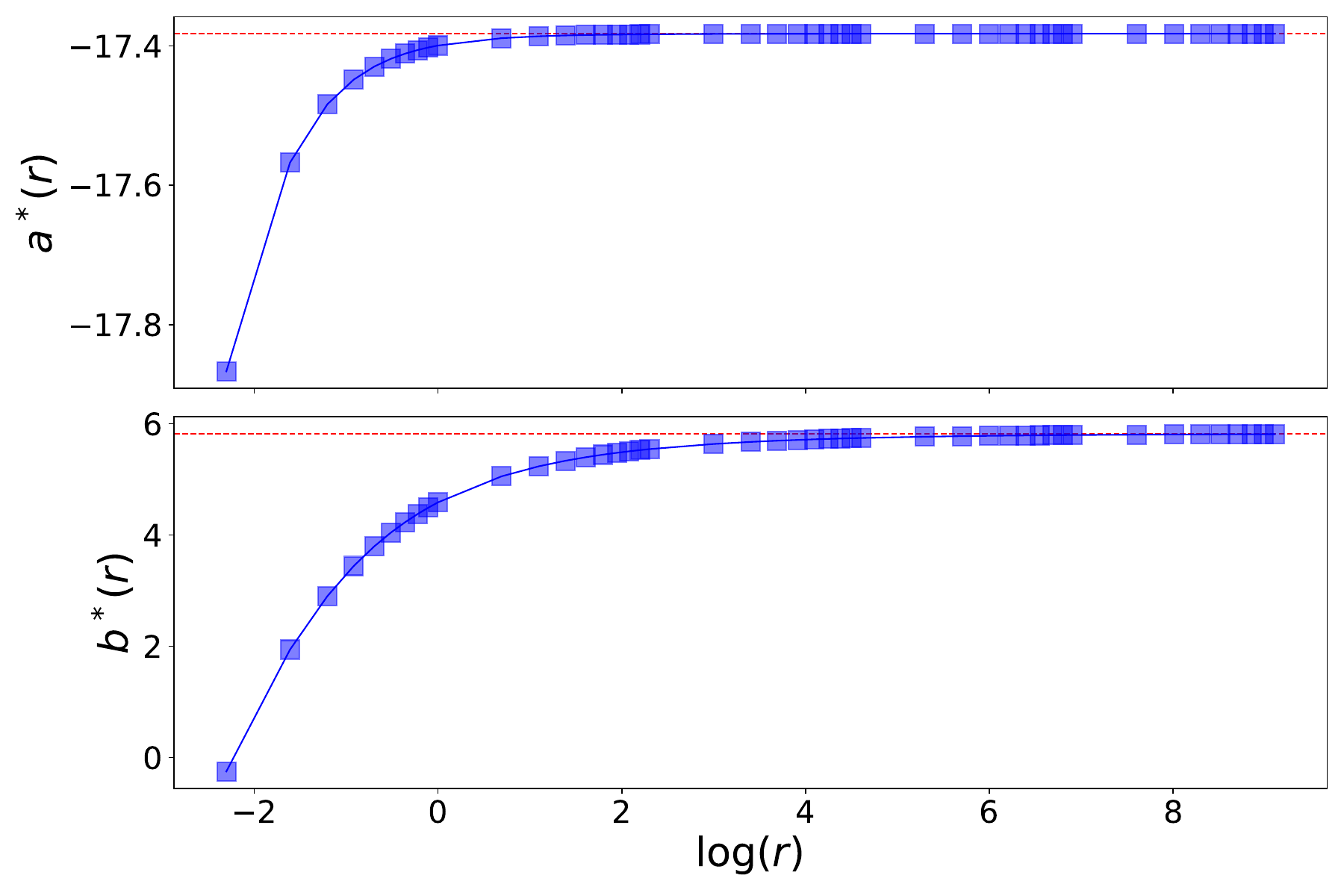}
\end{subfigure}
\begin{subfigure}[b]{0.49\textwidth}
    \centering
    \includegraphics[width=\textwidth]{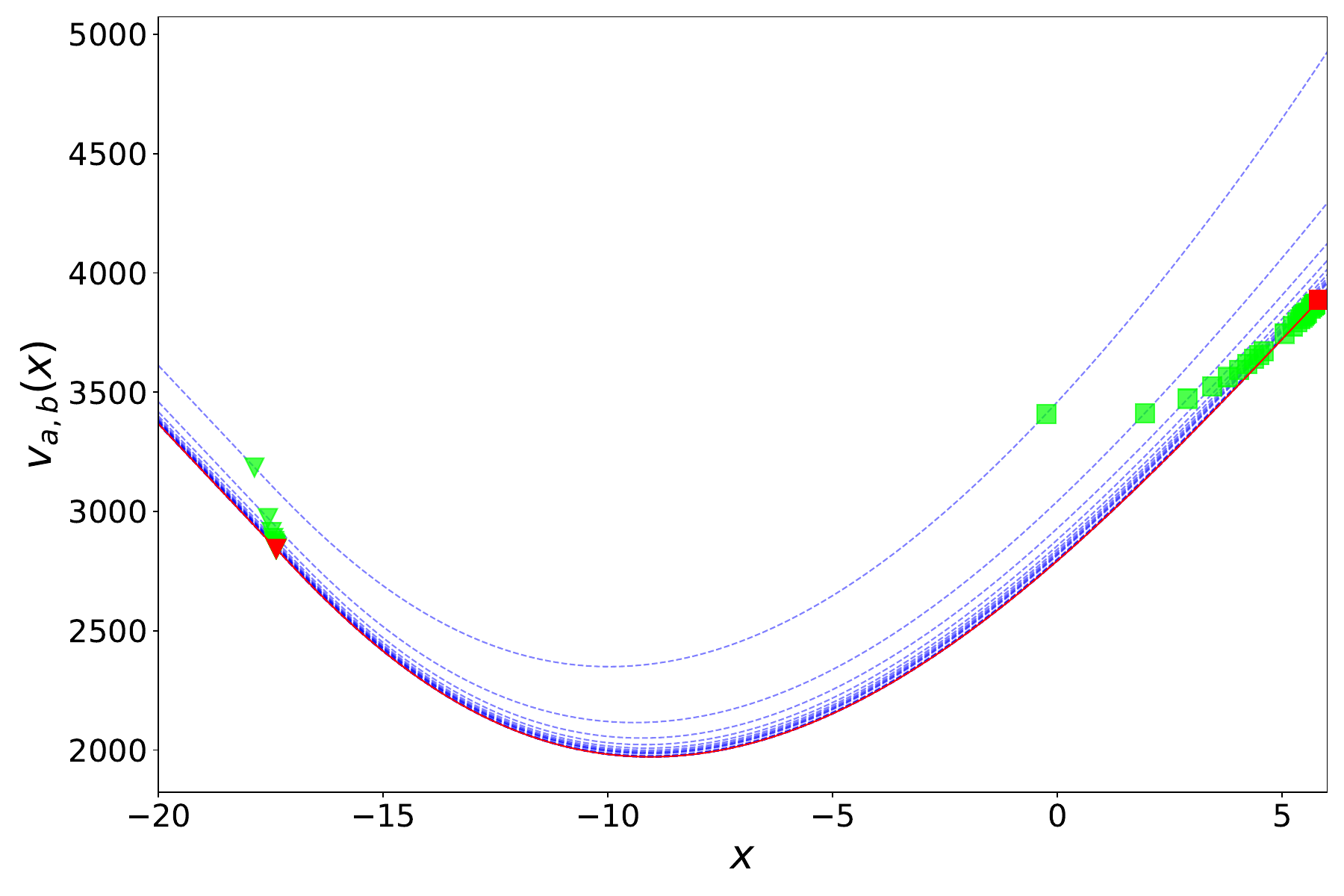}
\end{subfigure}
\caption{Left: Plot of periodic-classical barriers $(a^*, b^*)$ for $r = 0.1, 0.2, \dots, 0.9, 1, 2, \dots, 9, 10, 20, \break\dots, 90, 100, 200, \dots, 900$ (squares) and classical barriers (red dashed lines). Right: Plot of $v_{a^*, b^*}$ (blue dashed) for $r = 0.1, 0.2, \dots, 0.9, 1, 2, \dots, 9, 10, 20, \dots, 90$, with $(a^*, v_{a^*, b^*}(a^*))$ (lime triangles) and $(b^*, v_{a^*, b^*}(b^*))$ (lime squares), alongside the classical value function (red solid line), with lower and upper barriers marked by a red triangle and square, respectively.}
\label{Fig: limit r}
\end{figure}

The graph on the left in Figure \ref{Fig: limit r} shows the optimal periodic-classical barriers $a^*$ and $b^*$ as functions of arrival rates. These are compared with the classical barriers, which are the optimal barriers of the problem studied in \cite{baurdoux_optimality_2015}, the continuous-monitoring version of our current problem. As the arrival rate $r$ increases, both $a^*$ and $b^*$ converge to the classical barriers as obtained in \cite{baurdoux_optimality_2015}, indicating the convergence to the classical strategy.

Similarly, the graph on the right in Figure \ref{Fig: limit r} shows the value functions $v_{a^*, b^*}$ corresponding to various $r$ values, alongside the classical value function computed in \cite{baurdoux_optimality_2015}. As $r$ increases, $v_{a^*, b^*}$ decreases and converges pointwise to the classical value function.

\appendix

\section{Relaxation of Assumption \ref{asm: on f}(3)}\label{Sect: assumption on f (3)}
By Theorem \ref{Thm: optimal strategy}, when Assumption \ref{asm: on f} is satisfied, the double-barrier strategy with barriers $(a^*, b^*)$ is optimal, where $(a^*, b^*)$ is chosen such that $\mathfrak{C}$ holds. The finiteness of $b^*$ implies that activating the downward control is optimal. In cases where only Assumptions \ref{asm: on f}(1)--(2) hold and (3) is violated, the following result shows that not activating the downward control is optimal, and an optimal strategy is of the single-barrier type. 
\begin{proposition}\label{Prop: existence prop converse}
    Suppose Assumption \ref{asm: on f}(3) fails to hold (thus $f' \leq qC_D$ a.e. on $\mathbb{R}$) and all other standing assumptions including Assumptions \ref{asm: on f}(1)--(2) hold. Then, (i) not activating the downward control $(L^\pi(t); t \geq 0)$ is optimal; (ii) $v_{\underline{a}_2, \infty}(x) = \inf_{\pi \in \Pi} v^\pi(x)$ for all $x \in \mathbb{R}$, where $\underline{a}_2$ is defined in \eqref{Eq: a_underline}. 
\end{proposition}
\begin{proof}
    (i) holds since the downward control cost (resp., reward) is no less than (resp., no greater than) the reduction (resp., increase) in inventory cost. (ii) holds due to (i) and Theorem 2.1 of \cite{baurdoux_optimality_2015}.
\end{proof}

\section{Proofs}\label{Sect: proofs}
The following results are used for the proofs of Lemmas \ref{Lemma: inventory cost} and  \ref{Lemma: auxiliary result}.
\begin{proposition}[Proposition 5.1 of \cite{mata_bailout_2023}]\label{Prop: inventory cost computation}
    Suppose $h$ is a positive, bounded, measurable function on $\mathbb{R}$ with compact support. We have for $a < b$ and $x \in \mathbb{R}$,
    \begin{multline}\label{Eq: v_a, b at x}
        v_{a, b}^{h}(x) \coloneqq \mathbb{E}_x\left[\int^\infty_0e^{-qt} h(Y^{a, b}(t))\, \diff t\right] = \frac{\rho_{a, b}^{(q)}(b; h) + v_{a, b}^{h}(b)}{Z^{(q)}(b - a)} Z^{(q, r)}_{a, b}(x) - \rho_{a, b}^{(q, r)}(x; h)\\
        -r\overline{W}^{(q + r)}(x - b)v_{a, b}^{h}(b) - \int^{x}_b W^{(q + r)}(x - y)h(y)\, \diff y,
    \end{multline}
    where 
    \begin{multline}\label{Eq: v_a, b at b}
        v_{a, b}^{h}(b) = -\rho_{a, b}^{(q)}(b; h)\\
        + \left(\frac{q}{\Phi_{q + r}} \frac{Z^{(q)}(b - a; \Phi_{q + r})}{Z^{(q)}(b - a)}\right)^{-1}\left(-\frac{r}{\Phi_{q + r}}\rho_{a, b}^{(q)}(b; h) + \int^{\infty}_a h(y) Z^{(q)}(b - y, \Phi_{q + r})\, \diff y\right).
    \end{multline}
\end{proposition}

\begin{corollary}\label{Corollary: inventory cost computation}
    Proposition \ref{Prop: inventory cost computation} remains valid for $h: \mathbb{R} \to \mathbb{R}$ that satisfies Assumption \ref{asm: on f}(1) and $X$ satisfying Assumption \ref{asm: on X}. 
\end{corollary}
\begin{proof}
    Let $h^+ \coloneqq \max(h, 0)$ and $h^- \coloneqq \max(-h, 0)$ be the positive and negative parts of $h$, respectively. The positivity and boundedness assumptions on $h$ can be relaxed through the decomposition $h = h^+ - h^-$ and dominated convergence, by noting that $\mathbb{E}_x\left[\int^\infty_0 e^{-qt} h^+(Y^{a, b}(t))\, \diff t\right] < \infty$ and $\mathbb{E}_x\left[\int^\infty_0 e^{-qt} h^-(Y^{a, b}(t))\, \diff t\right] < \infty$. Indeed, the finiteness of these expectations holds because $h$ is of at most polynomial growth and $a \leq Y^{a,b}(t) \leq Y^{a,\infty}(t)$, where $Y^{a, \infty}$ is the classical reflected process from below at $a$ of the spectrally negative \lev process $X$, for which it is known that $\mathbb{E}_x\left[\int^\infty_0 e^{-qt + \theta Y^{a, \infty}(t)}\, \diff t\right] < \infty$ for any $\theta \geq 0$.
\end{proof}

\subsection{Proof of Lemma \ref{Lemma: inventory cost}} \label{proof_Lemma: inventory cost}
Corollary \ref{Corollary: inventory cost computation} can be applied due to Assumption \ref{asm: on f}(1). With \eqref{Eq: v_a, b at b}, via integration by parts as shown in the proof of Lemma 6.3 in \cite{mata_bailout_2023}, we obtain
    \begin{align*}
    \begin{split}
        v_{a, b}^{h}(b) = -\rho_{a, b}^{(q)}(b; h) + \frac{1}{q}\left(h(a) + \int^{\infty}_a h'(y) \frac{Z^{(q)}(b - y, \Phi_{q + r})}{Z^{(q)}(b - a, \Phi_{q + r})} \, \diff y\right)Z^{(q)}(b - a).
    \end{split}
    \end{align*}
    Substituting this into \eqref{Eq: v_a, b at x} and replacing $h$ with $f$, we have the desired expression.

\subsection{Proof of Lemma \ref{Lemma: NPV of costs}}\label{Sect: proof of NPV of costs}
Summing the inventory cost and the control costs, by Lemmas \ref{Lemma: control cost} and \ref{Lemma: inventory cost},
\begin{align}
\begin{split}\label{Eq: NPV of costs f}
    v_{a,b}(x) &= A \left(Z^{(q, r)}_{a, b}(x) - r\overline{W}^{(q + r)}(x - b)Z^{(q)}(b - a)\right)  +B - rC_D\overline{\overline{W}}^{(q + r)}(x - b)\\
    &- C_U\left(\overline{Z}^{(q, r)}_{a, b}(x) + \frac{\psi'(0+)}{q} - r\overline{W}^{(q + r)}(x - b)\overline{Z}^{(q)}(b - a) \right), \quad x \in \mathbb{R},
\end{split}
\end{align}
where 
\begin{align*}
    A &\coloneqq \frac{1}{q}\left(f(a) + \int^{\infty}_a f'(y)\frac{Z^{(q)}(b - y, \Phi_{q + r})}{Z^{(q)}(b - a, \Phi_{q + r})}\, \diff y + \frac{r}{\Phi_{q + r}} \frac{C_UZ^{(q)}(b - a) + C_D}{Z^{(q)}(b - a, \Phi_{q + r})} + \frac{qC_U}{\Phi_{q + r}}\right), \\
    B &\coloneqq  - \rho_{a, b}^{(q, r)}(x; f) +r\overline{W}^{(q + r)}(x - b)\rho_{a, b}^{(q)}(b; f) - \int^{x}_b f(y) W^{(q + r)}(x - y)\, \diff y.
\end{align*}
(i) First, we simplify the expression for $A$. 
\begin{lemma} \label{lemma_Z_integral}
    For $a< b$, $\int^\infty_a Z^{(q)}(b - y, \Phi_{q + r})\, \diff y =  (Z^{(q)}(b - a, \Phi_{q + r}) +  r\overline{W}^{(q)}(b - a)) / \Phi_{q+r}$.
\end{lemma}
\begin{proof}
By exchanging the order of integration and the second equality of \eqref{Eq: scale function Z},
\begin{multline*}
    \int^b_a Z^{(q)}(b - y, \Phi_{q + r})\, \diff y = r\int^b_a\int_0^{\infty} e^{-\Phi_{q + r} z} W^{(q)}(z + b - y)\, \diff z \, \diff y\\
    = r\int_0^{\infty} e^{-\Phi_{q + r} z} \int^b_a W^{(q)}(z + b - y)\,\diff y\,\diff z= r\int_0^{\infty} e^{-\Phi_{q + r} z}\left(\overline{W}^{(q)}(z + b - a) - \overline{W}^{(q)}(z)\right)\,\diff z.
\end{multline*}
By integration by parts and the second equality of \eqref{Eq: scale function Z},
\begin{align*}
    \int_0^{\infty} e^{-\Phi_{q + r} z} \overline{W}^{(q)}(z + b - a)\,\diff z = \frac{1}{\Phi_{q + r}}\overline{W}^{(q)}(b - a) + \frac{1}{r\Phi_{q+r}} Z^{(q)}(b - a, \Phi_{q + r}) .
\end{align*}
In particular, $\int_0^{\infty} e^{-\Phi_{q + r} z} \overline{W}^{(q)}(z)\, \diff z = (r\Phi_{q + r})^{-1} Z^{(q)}(0, \Phi_{q + r}) = (r\Phi_{q + r})^{-1}$. Thus, 
\begin{align*}
    \int^b_a Z^{(q)}(b - y, \Phi_{q + r})\, \diff y = \frac{1}{\Phi_{q+r}}\left(Z^{(q)}(b - a, \Phi_{q + r}) + r\overline{W}^{(q)}(b - a) - 1\right).
\end{align*}
In addition, $\int^\infty_b Z^{(q)}(b - y, \Phi_{q + r})\, \diff y = \int^\infty_b e^{\Phi_{q+r}(b - y)}\, \diff y = \int^\infty_0 e^{-\Phi_{q + r}y}\, \diff y = (\Phi_{q + r})^{-1}$. By summing these terms, the claim follows.
\end{proof}

By  Lemma \ref{lemma_Z_integral},
\begin{multline*}
    \int^{\infty}_a f'(y) Z^{(q)}(b - y, \Phi_{q + r})\, \diff y  \\
    = \int^{\infty}_a \tilde{f}'(y) Z^{(q)}(b - y, \Phi_{q + r})\, \diff y -  \frac {qC_U} {\Phi_{q+r}} \left(  Z^{(q)}(b - a, \Phi_{q + r}) +  r\overline{W}^{(q)}(b - a) \right).
\end{multline*}
Substituting this into the expression for $A$, we have 
\begin{equation} \label{A_simple}
    A = \frac{1}{q}\left(f(a) + \frac{\Gamma(a, b)}{Z^{(q)}(b - a, \Phi_{q + r})}\right). 
\end{equation}
(ii) Next, we write $B$ in terms of $\tilde{f}: x \mapsto f(x) + qC_Ux$, for $x \in \mathbb{R}$. First, 
\begin{align}\label{Eq: value misc 7}
    \rho_{a, b}^{(q)}(x; f) = \rho_{a, b}^{(q)}(x; \tilde{f}) - qC_U \rho_{a, b}^{(q)}(x; y \mapsto y),
\end{align}
where, due to integration by parts, 
\begin{align}
    \rho_{a, b}^{(q)}(x; y \mapsto y) &= \overline{W}^{(q)}(x - a) a - \overline{W}^{(q)}(x - b) b  + \overline{\overline{W}}^{(q)}(x - a) - \overline{\overline{W}}^{(q)}(x - b).\label{Eq: value misc 1}
\end{align}
In particular, at $x = b$, 
\begin{align}
    \rho_{a, b}^{(q)}(b; y \mapsto y) = \overline{W}^{(q)}(b - a) a + \overline{\overline{W}}^{(q)}(b - a).\label{Eq: value misc 2}
\end{align}
More generally, we have
\begin{align}
    \rho_{a, b}^{(q, r)}(x; f) = \rho_{a, b}^{(q, r)}(x; \tilde{f}) - q C_U \rho_{a, b}^{(q, r)}(x; y \mapsto y), \label{Eq: value misc 3}
\end{align}
where $\rho_{a, b}^{(q, r)}(x; y \mapsto y) = \rho_{a, b}^{(q)}(x; y \mapsto y) + r\int^{x}_b W^{(q + r)}(x - y)\rho_{a, b}^{(q)}(y; z \mapsto z)\, \diff y$.

Applying integration by parts twice, we obtain
\begin{multline}\label{Eq: value misc 6}
    \int^{x}_b f(y) W^{(q + r)}(x - y)\, \diff y\\
    = \int^{x}_b \tilde{f}'(y) \overline{W}^{(q + r)}(x - y)\, \diff y + f(b)\overline{W}^{(q + r)}(x - b) - qC_U\overline{\overline{W}}^{(q + r)}(x - b).
\end{multline}
Substituting \eqref{Eq: value misc 7}, \eqref{Eq: value misc 3}, and \eqref{Eq: value misc 6} into the expression for $B$, for $x \in \mathbb{R}$,  
\begin{align}
\begin{split}\label{Eq: value misc 0}
    B &= - \rho_{a, b}^{(q, r)}(x; \tilde{f}) + q C_U \rho_{a, b}^{(q, r)}(x; y \mapsto y) + r\overline{W}^{(q + r)}(x - b) \rho_{a, b}^{(q)}(b; \tilde{f}) - r\overline{W}^{(q + r)}(x - b) qC_U \rho_{a, b}^{(q)}(b; y \mapsto y)\\
    &~~ - \int^{x}_b \tilde{f}'(y) \overline{W}^{(q + r)}(x - y)\, \diff y - f(b)\overline{W}^{(q + r)}(x - b) + qC_U\overline{\overline{W}}^{(q + r)}(x - b). 
\end{split}
\end{align}
With
\begin{align*}
    \overline{W}^{(q, r)}_{a, b}(x) &\coloneqq \overline{W}^{(q)}(x - a) + r\int^x_b W^{(q + r)}(x - y)\overline{W}^{(q)}(y - a)\, \diff y, \\
    \overline{\overline{W}}^{(q, r)}_{a, b}(x) &\coloneqq \overline{\overline{W}}^{(q)}(x - a) + r\int^x_b W^{(q + r)}(x - y)\overline{\overline{W}}^{(q)}(y - a)\, \diff y,
\end{align*}
and using \eqref{Eq: value misc 1}, we have
\begin{align*}
    \rho_{a, b}^{(q, r)}(x; y \mapsto y) = a
    \overline{W}^{(q, r)}_{a, b}(x) + \overline{\overline{W}}^{(q, r)}_{a, b}(x) - b \overline{W}^{(q, r)}_{b, b}(x) - \overline{\overline{W}}^{(q, r)}_{b, b}(x).
\end{align*}
In terms of \eqref{Eq: Z_a, b} and \eqref{Eq: bar Z_a, b}, we can write
\begin{align*}
    \overline{W}^{(q, r)}_{a, b}(x) &= \frac{1}{q} \left( Z^{(q, r)}_{a, b}(x) - 1 - r\overline{W}^{(q + r)}(x - b) \right), \\
    \overline{\overline{W}}^{(q, r)}_{a, b}(x) &= \frac{1}{q} \left( \overline{Z}^{(q, r)}_{a, b}(x) - (x - a) - r\overline{W}^{(q + r)}(x - b)(b - a) -r\overline{\overline{W}}^{(q + r)}(x - b)\right). 
\end{align*}
Moreover, by Equation (6) of \cite{loeffen_occupation_2014}, we can write $\overline{W}^{(q, r)}_{b, b}(x) = \overline{W}^{(q + r)}(x - b)$. Then, by noticing that $\overline{\overline{W}}^{(q, r)}_{b, b}$ is an anti-derivative of $\overline{W}^{(q, r)}_{b, b}$ and using the identity $\overline{W}^{(q, r)}_{b, b}(x) = \overline{W}^{(q + r)}(x - b)$, we obtain $\overline{\overline{W}}^{(q, r)}_{b, b}(x) = \overline{\overline{W}}^{(q + r)}(x - b)$. Substituting these,
\begin{multline}
    \rho_{a, b}^{(q, r)}(x; y \mapsto y) = \frac 1 q \left(  a Z^{(q, r)}_{a, b}(x) +  \overline{Z}^{(q, r)}_{a, b}(x) - x - r \overline{W}^{(q + r)}(x - b)b - r\overline{\overline{W}}^{(q + r)}(x - b) \right) \\
    - \overline{\overline{W}}^{(q + r)}(x - b) - \overline{W}^{(q + r)}(x - b)b. \label{Eq: value misc 4}
\end{multline}
Substituting \eqref{Eq: value misc 2} and \eqref{Eq: value misc 4} in \eqref{Eq: value misc 0}, 
\begin{align*}
\begin{split}
    B &= - \rho_{a, b}^{(q, r)}(x; \tilde{f}) + C_U   \left(  a Z^{(q, r)}_{a, b}(x) +  \overline{Z}^{(q, r)}_{a, b}(x) - x - r \overline{W}^{(q + r)}(x - b)b - r\overline{\overline{W}}^{(q + r)}(x - b) \right) \\
    &~~ + r\overline{W}^{(q + r)}(x - b) \rho_{a, b}^{(q)}(b; \tilde{f}) - r\overline{W}^{(q + r)}(x - b) qC_U \left( \overline{W}^{(q)}(b - a) a + \overline{\overline{W}}^{(q)}(b - a) \right) \\
    &~~ - \int^{x}_b \tilde{f}'(y) \overline{W}^{(q + r)}(x - y)\, \diff y - \tilde{f}(b)\overline{W}^{(q + r)}(x - b). 
\end{split}
\end{align*}
(iii) Substituting this expression for $B$ and \eqref{A_simple} into \eqref{Eq: NPV of costs f}, and simplifying, we obtain the desired expression.

\subsection{Proof of Lemma \ref{Lemma: smoothness lemma}}\label{Sect: proof of smoothness lemma}
By direct computation and the smoothness of the scale functions, we obtain the following results for $x \in \mathbb{R}\backslash \{a, b\}$ and for any continuous and a.e.-differentiable function $h: \mathbb{R} \to \mathbb{R}$:
\begin{align}
    \frac{\diff}{\diff x} Z^{(q, r)}_{a, b}(x) &= qW^{(q, r)}_{a, b}(x) + rW^{(q + r)}(x - b) Z^{(q)}(b - a),\label{Eq: Z_a, b derivative}\\
    \frac{\diff}{\diff x} \rho_{a, b}^{(q)}(x; h) &= -W^{(q)}(x - b)h(b) + W^{(q)}(x - a)h(a) + \int^{b}_a W^{(q)}(x - y)h'(y)\, \diff y.\label{Eq: rho derivative}
\end{align}
We further compute $\frac{\diff}{\diff x}\rho_{a, b}^{(q, r)}(x; h)$. Recalling \eqref{Eq: rho_r}, by partial integration, 
\begin{align}\label{Eq: rho_r derivative part 1}
    \frac{\diff}{\diff x} \rho_{a, b}^{(q, r)}(x; h) = \frac{\diff}{\diff x}\rho_{a, b}^{(q)}(x; h) +  rW^{(q + r)}(x - b)\rho_{a, b}^{(q)}(b; h) + C, \quad x\in \mathbb{R}\backslash\{a, b\},
\end{align}
where, by \eqref{Eq: rho derivative} and Equation (6) of \cite{loeffen_occupation_2014},
\begin{multline}\label{Eq: misc 7}
    C \coloneqq r\int^{x}_{b} W^{(q + r)}(x - y)\frac{\diff}{\diff y}\rho_{a, b}^{(q)}(y; h)\, \diff y = -h(b)\left(W^{(q + r)}(x - b) - W^{(q)}(x - b)\right)\\
    + rh(a)\int^{x}_{b} W^{(q + r)}(x - y)W^{(q)}(y - a)\, \diff y + r\int^{x}_{b} W^{(q + r)}(x - y) \int^{b}_a W^{(q)}(y - z)h'(z)\, \diff z\, \diff y. 
\end{multline}
Substituting \eqref{Eq: misc 7} in \eqref{Eq: rho_r derivative part 1} and using \eqref{Eq: W_b, a}, for $x\in \mathbb{R}\backslash\{a, b\}$,
\begin{multline}\label{Eq: rho_r derivative part 2}
    \frac{\diff}{\diff x} \rho_{a, b}^{(q, r)}(x; h) = h(a)W^{(q, r)}_{a, b}(x) + rW^{(q + r)}(x - b)\rho_{a, b}^{(q)}(b; h) - h(b)W^{(q + r)}(x - b) + \rho_{a, b}^{(q, r)}(x; h').
\end{multline}
We obtain \eqref{Eq: value derivative} by differentiating \eqref{Eq: NPV of costs} using \eqref{Eq: Z_a, b derivative} and \eqref{Eq: rho_r derivative part 2}. The statement that \eqref{Eq: value derivative} is defined on $\mathbb{R}\backslash \{a\}$ can be confirmed by inspecting its left- and right-hand limits and using the smoothness of scale functions (see Remark \ref{remark_smoothness_scale_function}). To obtain the second derivative of $v_{a, b}$, we directly differentiate \eqref{Eq: value derivative}. This gives \eqref{Eq: value 2nd derivative}, which is again defined on $\mathbb{R}\backslash \{a\}$ by the smoothness of scale functions. 

\subsection{Proof of Lemma \ref{Lemma: auxiliary result}}\label{Sect: proof of auxiliary lemma}
By \eqref{Eq: rho} and the definition of $\tilde{f}$, for $x \in \mathbb{R}$,
\begin{align}\label{Eq: rho_r f' f_tilde'}
    \rho_{a, b}^{(q, r)}(x; f') &= \rho_{a, b}^{(q,r)}(x; \tilde{f}') - q C_U \rho_{a, b}^{(q,r)}(x; 1),
\end{align}
where, with $\rho_{a, b}^{(q)}(x; 1) = \int^{b}_a W^{(q)}(x - y) \, \diff y = \frac{1}{q}(Z^{(q)}(x - a) - Z^{(q)}(x - b))$, we can write
\begin{align*}
    \rho_{a, b}^{(q,r)}(x; 1) &= \rho_{a, b}^{(q)}(x; 1) + r\int^{x}_b W^{(q + r)}(x - y)\rho_{a, b}^{(q)}(y; 1) \diff y= \frac{1}{q} \left(Z^{(q, r)}_{a, b}(x) - Z^{(q, r)}_{b, b}(x)\right).
\end{align*}
Using Equation (6) of \cite{loeffen_occupation_2014}, we obtain $Z^{(q, r)}_{b, b}(x) = Z^{(q + r)}(x - b)$, and thus,
\begin{align}
    \rho_{a, b}^{(q, r)}(x; f') &= \rho_{a, b}^{(q, r)}(x; \tilde{f}') - C_U\left(Z^{(q, r)}_{a, b}(x) - Z^{(q + r)}(x - b)\right).\label{Eq: rho_r f' f_tilde'}
\end{align}
Moreover, for $x \in \mathbb{R}$, we have 
\begin{align}
\begin{split}\label{Eq: value derivative misc 2}
    \int^{x}_{b} W^{(q + r)}(x - y)f'(y)\, \diff y &= \int^{x}_{b} W^{(q + r)}(x - y)\tilde{f}'(y)\, \diff y - qC_U\overline{W}^{(q + r)}(x - b). 
\end{split}
\end{align}
Rewriting \eqref{Eq: v_a, b at b} with $h = f'$ by applying Lemma \ref{lemma_Z_integral} and \eqref{Eq: rho_r f' f_tilde'}, and simplifying using \eqref{Eq: gamma function}, we obtain
\begin{align} \label{Eq: v^f' misc 3}
\begin{split}
    v_{a, b}^{f'}(b) &= -\rho_{a, b}^{(q)}(b; \tilde{f}') + qC_U\overline{W}^{(q)}(b - a) + \left(\frac{q}{\Phi_{q + r}} \frac{Z^{(q)}(b - a; \Phi_{q + r})}{Z^{(q)}(b - a)}\right)^{-1}\\
    &~~ \times \left(-\frac{r}{\Phi_{q + r}}\rho_{a, b}^{(q)}(b; \tilde{f}') + \int^{\infty}_a \tilde{f}'(y) Z^{(q)}(b - y, \Phi_{q + r})\, \diff y - \frac{qC_U}{\Phi_{q + r}} Z^{(q)}(b - a, \Phi_{q + r})\right)\\
    &= -\rho_{a, b}^{(q)}(b; \tilde{f}') - C_U + \frac{Z^{(q)}(b - a)}{qZ^{(q)}(b - a; \Phi_{q + r})}\gamma(a, b). 
    \end{split}
\end{align}
Substituting \eqref{Eq: rho_r f' f_tilde'}, \eqref{Eq: value derivative misc 2}, and \eqref{Eq: v^f' misc 3} into \eqref{Eq: v_a, b at x} with $h = f'$, and simplifying, we obtain the desired result.

\subsection{Proof of Lemma \ref{Lemma: F asymptotics}}\label{Sect: proof of F asymptotics}
(i): As $b \mapsto \Gamma(a, b)$ is continuous on $(a, \infty)$, the first claim holds by setting $b = a+$ in \eqref{Eq: Gamma function} and noting that $Z^{(q)}(a - y, \Phi_{q + r}) = \exp(\Phi_{q + r}(a - y))$ for $y > a$.

\noindent (ii): By the second equality of \eqref{Eq: scale function Z}, we have two bounds for $Z^{(q)}(\cdot, \Phi_{q + r})$. First, with $W_{\Phi_{q}}(x) = e^{-\Phi_q x}W^{(q)}(x)$ (see page 247 of \cite{kyprianou_fluctuations_2014}), for $x \in \mathbb{R}$,
\begin{equation*}
    Z^{(q)}(x, \Phi_{q + r}) = r\int^\infty_0 e^{-\Phi_{q + r} y}W^{(q)}(x + y)\, \diff y = re^{\Phi_{q} x}\int^\infty_0 e^{-(\Phi_{q + r} - \Phi_{q}) y}W_{\Phi_{q}}(x + y)\, \diff y.
\end{equation*}
Since $W_{\Phi_q}$ is increasing, $\frac{re^{\Phi_{q} x} W_{\Phi_{q}}(x)}{\Phi_{q + r} - \Phi_{q}} \leq Z^{(q)}(x, \Phi_{q + r}) \leq \frac{re^{\Phi_{q} x} W_{\Phi_{q}}(\infty)}{\Phi_{q + r} - \Phi_{q}}$, where $W_{\Phi_{q}}(\infty) \coloneqq \lim_{x \to \infty} W_{\Phi_{q}}(x)$ is well-defined and finite. Thus, for $b > b'$ with $b' > a$ fixed,
\begin{align*}
    \left| \tilde{f}'(z) \frac{Z^{(q)}(b - z, \Phi_{q + r})}{Z^{(q)}(b - a, \Phi_{q + r})} 1_{\{a \leq z\}} \right| &\leq \left|\tilde{f}'(z)\right| \frac{e^{\Phi_q (b - z)} W_{\Phi_{q}}(\infty)}{e^{\Phi_q (b - a)} W_{\Phi_{q}}(b - a)} \leq \left|\tilde{f}'(z)\right| e^{-\Phi_q (z - a)}\frac{ W_{\Phi_{q}}(\infty)}{W_{\Phi_{q}}(b' - a)},
\end{align*}
which is integrable. Thus, dominated convergence gives
\begin{align*}
    \lim_{b \to \infty} \frac{\Gamma(a, b)}{Z^{(q)}(b - a, \Phi_{q + r})} = \int^{\infty}_a \lim_{b\to\infty} \tilde{f}'(z) \frac{Z^{(q)}(b - z, \Phi_{q + r})}{Z^{(q)}(b - a, \Phi_{q + r})}\, \diff z &= \int^{\infty}_0 e^{-\Phi_q z} \tilde{f}'(z + a)\, \diff z.
\end{align*}

\subsection{Proof of Lemma \ref{Lemma: verification lemma}}\label{Sect: proof of verification lemma}
Fix an admissible strategy $\pi \in \Pi$. For a c\`adl\`ag process $A$, denote its jump at time $s$ by $\Delta A(s) \coloneqq A(s) - A(s-)$ and its continuous part by $A^c$, such that $A(s) = A^c(s) + \sum_{0 \leq u \leq s}\Delta A(u)$. By first applying It\^o's lemma and then rearranging terms, we obtain
\begin{align*}
    w(x) &= -\int^{t \wedge \tau_n}_0 e^{-qs} (\mathcal{L}-q) w(Y^\pi(s-))\, \diff s\\
    &~~ - \int_{[0, t \wedge \tau_n]} e^{-qs}w'(Y^\pi(s-))\, \diff R^{\pi,c}(s) + C_D \int_{[0, t \wedge \tau_n]} e^{-qs} \nu^\pi(s)\, \diff N^r(s)\\
    &~~ - \int^{t \wedge \tau_n}_0 re^{-qs} [C_D \nu^\pi(s) + w(Y^\pi(s-) - \nu^\pi(s)) - w(Y^\pi(s-))]\, \diff s\\
    &~~ - \sum_{0\leq s \leq t \wedge \tau_n} e^{-qs}[ w(Y^\pi(s-) + \Delta X(s) + \Delta R^\pi(s)) - w(Y^\pi(s-) + \Delta X(s))]\\
    &~~ - M(t \wedge \tau_n) + e^{-q(t\wedge \tau_n)} w(Y^\pi(t\wedge \tau_n)),
\end{align*}
where $(M(t \wedge \tau_n); t \geq 0)$ is defined in Equation A.1 in \cite{perez_optimality_2017}. Noting that $Y^\pi(s-)$ is bounded a.s on $[0, t \wedge \tau_n]$ and using Corollary 4.6 from \cite{kyprianou_fluctuations_2014}, we have that the process $M(t \wedge \tau_n)$ is a zero-mean $\mathbb{P}_x$-martingale. Now, since $-C_U \leq w'(x)$ and $(\mathcal{L}-q)w(x) + r(\mathcal{M}w(x) - w(x)) + f(x) \geq 0$ for all $x \in \mathbb{R}$ by assumption, 
\begin{align*}
    w(x) &\leq -\int^{t \wedge \tau_n}_0 e^{-qs} \left[(\mathcal{L}-q) w(Y^\pi(s-)) + r(\mathcal{M}w(Y^\pi(s-)) - w(Y^\pi(s-)))\right]\, \diff s\\
    &~~ + C_D \int_{[0, t \wedge \tau_n]} e^{-qs} \nu^\pi(s)\, \diff N^r(s) + C_U\int_{[0, t \wedge \tau_n]} e^{-qs}\, \diff R^{\pi}(s)\\
    &~~ - M(t \wedge \tau_n) + e^{-q(t\wedge \tau_n)} w(Y^\pi(t\wedge \tau_n))\\
    &\leq \int^{t \wedge \tau_n}_0 e^{-qs} f(Y^\pi(s))\, \diff s + C_D \int_{[0, t \wedge \tau_n]} e^{-qs} \nu^\pi(s)\, \diff N^r(s) + C_U\int_{[0, t \wedge \tau_n]} e^{-qs}\, \diff R^{\pi}(s)\\
    &~~ - M(t \wedge \tau_n) + e^{-q(t\wedge \tau_n)} w(Y^\pi(t\wedge \tau_n)).
\end{align*}
Taking expectation, we obtain
\begin{multline}\label{Eq: misc 10}
    w(x) \leq \mathbb{E}_x\left[\int^{t \wedge \tau_n}_0 e^{-qs} f(Y^\pi(s))\, \diff s + C_D \int_{[0, t \wedge \tau_n]} e^{-qs} \nu^\pi(s)\, \diff N^r(s) + C_U\int_{[0, t \wedge \tau_n]} e^{-qs}\, \diff R^{\pi}(s)\right]\\
    + \mathbb{E}_x\left[e^{-q(t\wedge \tau_n)} w(Y^\pi(t\wedge \tau_n))\right].
\end{multline}
We take $t, n \uparrow \infty$ to complete the proof. As assumed in the lemma, for any admissible strategy $\pi$, $\limsup_{t, n \uparrow \infty} \mathbb{E}_x\left[e^{-q(t\wedge \tau_n)} w(Y^\pi(t\wedge \tau_n))\right] \leq 0$. Taking limit of the other terms on the right-hand side of \eqref{Eq: misc 10}, because $\pi$ is admissible, dominated convergence gives 
\begin{align*}
    w(x) &\leq \mathbb{E}_x\left[\int^{\infty}_0 e^{-qs} f(Y^\pi(s))\, \diff s + \int_{[0, \infty)} e^{-qt}(C_U \, \diff R^\pi(t) + C_D \, \diff L^\pi(t))\right] = v^\pi(x).
\end{align*}
The proof is now complete, as $\pi \in \Pi$ was chosen arbitrarily.

\subsection{Proof of Lemma \ref{Lemma: generator M}}\label{Sect: proof of verification generator}
(i) Applying $\Gamma(a^*,b^*)=0$ in  \eqref{Eq: NPV of costs}, for $x \leq a^*$, notice that $v_{a^*, b^*}(x) = (-C_U\psi'(0+) + \tilde{f}(a^*))/q - C_Ux$. It follows that $(\mathcal{L}-q)v_{a^*, b^*}(x) + f(x) = \tilde{f}(x) - \tilde{f}(a^*)$, which is non-negative. This shows the first statement of the lemma. \\
\noindent (ii) For $x > a^*$, the same argument as in Lemma 7.2 of \cite{mata_bailout_2023} applies, since the form of $v_{a^*, b^*}^{LR}$ in terms of scale functions is identical to that considered in Lemma 7.2 of \cite{mata_bailout_2023}. Therefore, for $x > a^*$, we have, 
\begin{align*}
    (\mathcal{L} - q) v_{a^*, b^*}^{LR}(x) &= \begin{cases}
        0, & a^* < x < b^*,\\
        -r\left(C_D(x - b^*) + v_{a^*, b^*}^{LR}(b^*) - v_{a^*, b^*}^{LR}(x)\right), & x \geq b^*\\
    \end{cases} \\
    (\mathcal{L} - q) v_{a^*, b^*}^{f}(x) &= \begin{cases}
        -f(x), & a^* < x < b^*,\\
        -r\left(v_{a^*, b^*}^{f}(b^*) - v_{a^*, b^*}^{f}(x)\right) - f(x), & x \geq b^*.\\
    \end{cases}
\end{align*}
Summing these, we get the desired expression.

\subsection{Proof of Lemma \ref{Lemma: verification limit}}\label{Sect: proof of verification limit}
We first establish the following auxiliary result. Define
\begin{equation}\label{Eq: g function}
    g(x) \coloneqq \mathbb{E}_x \left[\int^\infty_0 e^{-qt} f(X(t))\, \diff t\right] = q^{-1} \mathbb{E}\left[f(X(e_q) + x) \right],
\end{equation}
    where $e_q$ is an independent exponential random variable with rate $q$.
\begin{lemma}\label{Lemma: reg/slow varying}
    Suppose $f$ is a convex and slowly or regularly varying function at $+\infty$ (resp., $-\infty$), with $|f(x)| \to \infty$ as  $x\to +\infty$ (resp., $x\to -\infty$). Then, we have $\lim_{x \uparrow \infty} |g(x)/f(x)| = 1/q$ (resp., $\lim_{x \downarrow -\infty} |g(x)/f(x)| = 1/q$).
\end{lemma}

\begin{proof}
We show the claim for $x \to \infty$; the other case can be treated analogously. For such $f$, there exists some $c \in \mathbb{R}$ such that, for all $\alpha > 0$, 
\[u(\alpha) \coloneqq \lim_{x \to \infty} \frac{f(\alpha x)}{f(x)} = \begin{cases} 1, & \text{if $f$ is slowly varying at $\infty$,}\\
\alpha^c, & \text{if $f$ is regularly varying at $\infty$.}
\end{cases}\]
Fix some $1 < a_0 < 2$ and $\varepsilon > 0$, by the uniform convergence theorem for regularly varying functions (see Theorem 1.5.2) in \cite{Bingham_Goldie_Teugels_1987}), there exists some $x_0 > 0$ such that
\[\sup_{a \in [1, a_0], x > x_0}\left|\frac{f(ax)}{f(x)} - u(a)\right| < \varepsilon.\]

Define
\begin{align} \label{def_x_1}
    x_1 \coloneqq \inf\{x \in \mathbb{R}: |f| \geq 1 \text{ and $|f|$ is strictly increasing on $(x, \infty)$}\} < \infty.
\end{align}
For all $x > x_0 \vee x_1, y \in [0, (a_0 - 1)x]$,
\[
1 \leq \left| \frac {f(y + x)} {f(x)} \right| \leq \left|  \frac {f(a_0 x)} {f(x)}\right| \leq u(a_0) + \varepsilon.
\]
Hence, setting $M_1 \coloneqq u(a_0) + \varepsilon$, we have
\[ \sup_{x > x_0 \vee x_1, y \in [0, (a_0 - 1)x]} \left|\frac{f(y + x)}{f(x)}\right| < M_1 < \infty. \]

Consider the following integral, which depends on the value of $x$,
\[\mathbb{E}\left[\frac{f(X(e_q) + x)}{f(x)}; X(e_q) \geq 0\right].\]
We seek an integrable function that bounds the integrand for all sufficiently large $x$. Fix $x > x_0 \vee x_1$, for $y \in [0, (a_0 - 1)x]$, we have from the analysis above that $|f(y + x)/f(x)| \leq M_1$. For $y > (a_0 - 1)x$, $f(y + x)/f(x)$ can be bounded by a polynomial function, as
\[\left|\frac{f(x + y)}{f(x)}\right| < |f(x + y)| < \left|f \left(\frac y {a_0-1} + y \right) \right| = \left|f\left(y\frac{a_0}{a_0 - 1}\right)\right|,\]
where the first inequality follows as $x > x_1$ and the second inequality follows as $x > x_1$ and $y > (a_0 - 1)x$. Thus, for $y \geq 0$, we have
\begin{align*}
    \sup_{x > x_0 \vee x_1}\left|\frac {f(y+x)} {f(x)}\right| \leq M_1 \vee f\left(y\frac{a_0}{a_0 - 1} \right) \eqqcolon h(y),
\end{align*}
where $h$ is of polynomial growth such that $\mathbb{E} [h(X(e_q))] < \infty$ by Assumption \ref{asm: on X}. Hence, by dominated convergence, we obtain
\[\lim_{x\to\infty} \mathbb{E}\left[\frac{f(X(e_q) + x)}{f(x)}; X(e_q) \geq 0\right] = \mathbb{P}(X(e_q) \geq 0).\]

Now, we consider the case with the event $\{X(e_q) < 0\}$. Let $m$ denote the minimum value of $f$ if it exists (i.e. if $f$ is not monotone); otherwise, set $m$ to zero. By convexity, for all $y \leq 0$ and $x \geq 0$,
\[
|f(x+y)| \leq |m| \vee |f(y)| \vee |f(x)|.
\]
For $x > x_1$ (as in \eqref{def_x_1}), we have  $|f(x+y)/f(x)| \leq |m| \vee |f(y)| \vee 1$.

Thus, there exists some $x_2 > 0$ and a function $\tilde{h}: \mathbb{R} \to \mathbb{R}_+$ with $\mathbb{E} [\tilde{h}(X(e_q))] < \infty$, such that for all $y \in \mathbb{R}$, we have $\sup_{x > x_2}|f(y+x)/f(x)| \leq \tilde{h}(y)$. Applying dominated convergence, we obtain $\lim_{x\to\infty} \mathbb{E}[f(X(e_q) + x)/f(x); X(e_q) < 0] = \mathbb{P}(X(e_q) < 0)$.
\end{proof}

Now, we are ready to show Lemma \ref{Lemma: verification limit}. Following the same argument as the proof of Lemma 7.5 in \cite{yamazaki_inventory_2017}, we have the following upper bound:
\begin{align}\label{Eq: upper bound}
\begin{split}
    \mathbb{E}_x\left[e^{-q(t\wedge \tau_n)} w(Y^\pi(t\wedge \tau_n))\right] &\leq \mathbb{E}_x\left[ e^{-q(t \wedge \tau_n)} \mathbb{E}_{Y^\pi(t \wedge \tau_n)} \left[\int^\infty_{0}e^{-qs}f(X(s))\, \diff s\right]\right]\\
    &= \mathbb{E}_x\left[ e^{-q(t \wedge \tau_n)} g(Y^\pi(t \wedge \tau_n))\right].
    \end{split}
\end{align}
We show
\begin{align} \label{g_bounded}
    \exists M, N < \infty \quad \textrm{such that } |g(x)| \leq \max(M|f(x)|, N).
\end{align} 
(i) If $|f(x)| \to \infty$ as $x \to +\infty$ and also as $x \to -\infty$, then, by Lemma \ref{Lemma: reg/slow varying}, \eqref{g_bounded} holds.

\noindent (ii) Let $|f|$ be bounded as $x\to -\infty$ and slowly or regularly varying at $+\infty$, with $|f|$ unbounded as $x \to +\infty$. Again using Lemma \ref{Lemma: reg/slow varying}, for some positive $M < \infty$ and $x' \in \mathbb{R}$, $|g(x)|\leq M|f(x)|$ for $x > x'$. Since $|f|$ is bounded as $x \to -\infty$, and $f$ is convex, it follows that $f$ is a monotonically increasing function. Consequently, for $x < x'$, we have $g(x') \geq g(x) \geq q^{-1} \mathbb{E}\left[f(-\infty) \right]$, which is finite. Thus, we again establish \eqref{g_bounded}.

\noindent (iii) The case where $|f|$ is bounded as $x \to +\infty$ and slowly or regularly varying at $-\infty$, with $|f|$ unbounded as $x \to -\infty$, follows analogously to (ii). 

\noindent (iv) Finally, Assumption \eqref{asm: on f} excludes the case $f$ is bounded on $\mathbb{R}$. Hence, in all cases, \eqref{g_bounded} holds.

To complete the proof, we show that the right-hand side of \eqref{Eq: upper bound} limits to $0$. Since $\pi$ is admissible, we have $\mathbb{E}_x[\int^\infty_{0}e^{-qs} |f(Y^\pi(s))|\, \diff s] < \infty$, thus $\lim_{t \uparrow \infty}  \mathbb{E}_x\left[e^{-qt} f(Y^\pi(t))\right] = 0$. Then, $\lim_{t \uparrow \infty} \mathbb{E}_x\left[e^{-qt} g(Y^\pi(t))\right] = 0$ follows by the bound in \eqref{g_bounded}. This, in conjunction with \eqref{Eq: upper bound}, completes the proof of Lemma \ref{Lemma: verification limit}.

\bibliographystyle{abbrv}
\bibliography{main}
\end{document}